\documentclass{amsart}
\usepackage{graphicx}
\usepackage{cite}
\usepackage{amssymb}
\newtheorem{thm}{Theorem}[section]
\newtheorem{cor}[thm]{Corollary}
\newtheorem{lem}[thm]{Lemma}
\newtheorem{prop}[thm]{Proposition}
\theoremstyle{definition}
\newtheorem{defn}[thm]{Definition}
\theoremstyle{remark}
\newtheorem{rem}[thm]{Remark}
\numberwithin{equation}{section}

\begin{document}

\title[Poisson kernels and $L^p$ polyharmonic boundary value problems]{Higher order Poisson Kernels and $L^p$ polyharmonic boundary value problems in Lipschitz domains}%
\author{Zhihua Du}%
\address{Department of Mathematics, Jinan University, Guangzhou 510632, China}%
\email{tzhdu@jnu.edu.cn}%


\begin{abstract}
In this article, we introduce higher order conjugate Poisson and
Poisson kernels, which are higher order analogues of the
classical conjugate Poisson and Poisson kernels, as well as the polyharmonic fundamental solutions, and define
multi-layer potentials in terms of Poisson field and the polyharmonic fundamental solutions, in which the former formed by the
higher order conjugate Poisson and Poisson kernels. Then by the
multi-layer potentials, we solve three classes of boundary value
problems (i.e., Dirichlet, Neumann and regularity problems) with $L^{p}$
boundary data for polyharmonic equations in Lipschitz domains and
give integral representation (or potential) solutions of these problems.
\end{abstract}
\maketitle
\section{Introduction}

Let $D$ be a Lipschitz graph domain or bounded Lipschitz domain in $\mathbb{R}^{n+1}$, $n\geq 2$. In
this work, we will resolve the following boundary value problems (simply, BVPs) for
polyharmonic functions in $D$ with $L^{p}$ boundary data:

\begin{description}
  \item[Dirichlet problem]  \begin{equation}\begin{cases}\Delta^{m}u=0,\,\,in\,\, D,\\
  \Delta^{j}u=f_{j},\,\,on\,\,\partial D,\\
  \big(u-M_{1}\widetilde{f}_{0}\big)\in L^{p}(D)
  \end{cases}
  \end{equation}
  with $\|u-M_{1}\widetilde{f}_{0}\|_{L^{p}(\partial D)}\leq C\sum_{j=1}^{m-1}\|f_{j}\|_{L^{p}(\partial D, wd\sigma)}$, where $\Delta$ is the Laplacian, $f_{0}\in L^{p}(\partial D)$, $f_{j}\in L^{p}(\partial D, wd\sigma)$, $1\leq j\leq m-1$ for some $p\in (1,\infty)$ and some certain weight functions $w$ on $\partial D$ (if $D$ is bounded, $w\equiv 1$ ), $d\sigma$ is the area measure of $\partial D$, $\widetilde{f}_{0}$ is related to all the boundary data $f_{j}$, $0\leq j\leq m-1$, $M_{1}$ is the classical double layer potential operator, and the constant $C$ depends only on $m,n,p$ and $D$.
  \item[Neumann problem] \begin{equation}\begin{cases}\Delta^{m}u=0,\,\,in\,\, D,\\
  \frac{\partial}{\partial N}\Delta^{j}u=g_{j},\,\,on\,\,\partial
  D,\\
  \nabla\big(u-\mathcal{M}_{1}\widetilde{g}_{0}\big)\in L^{p}(D)
  \end{cases}
  \end{equation}
  with $\|\nabla\big(u-\mathcal{M}_{1}\widetilde{g}_{0}\big)\|_{L^{p}(\partial D)}\leq C\sum_{j=1}^{m-1}\|g_{j}\|_{L^{p}(\partial D, wd\sigma)}$, where $\Delta$ is the Laplacian, $\nabla$ is the gradient operator, $\frac{\partial}{\partial N}$ denotes the
outward normal derivative, $g_{0}\in L^{p}(\partial D)$, $g_{j}\in L^{p}(\partial D, wd\sigma)$, $1\leq j\leq m-1$ for some $p\in (1,\infty)$ and some certain weight functions $w$ on $\partial D$ (if $D$ is bounded, $w\equiv 1$, and $g_{m-1}$ has mean value zero, i.e. $\int_{\partial D}g_{m-1}d\sigma=0$), $d\sigma$ is the area measure of $\partial D$, $\widetilde{g}_{0}$ is related to all the boundary data $g_{j}$, $0\leq j\leq m-1$, $\mathcal{M}_{1}$ is the classical single layer potential operator, and the constant $C$ depends only on $m,n,p$ and $D$.
  \item[Regularity problem] \begin{equation}\begin{cases}\Delta^{m}u=0,\,\,in\,\, D,\\
  \Delta^{j}u=h_{j},\,\,on\,\,\partial D,\\
  \nabla\big(u-\mathcal{M}_{1}\widetilde{h}_{0}\big)\in L^{p}(D)
  \end{cases}
  \end{equation}
  with $\| \nabla\big(u-\mathcal{M}_{1}\widetilde{h}_{0}\big)\|_{L^{p}(D)}\leq C\sum_{j=1}^{m-1}\|h_{j}\|_{L^{p}_{1}(\partial D, wd\sigma)}$, where $\Delta$ is the Laplacian, $\nabla$ is the gradient operator, $h_{0}\in L_{1}^{p}(\partial D)$, $h_{j}\in L_{1}^{p}(\partial D, wd\sigma)$, $0\leq j\leq m-1$ for some $p\in (1,\infty)$ and some certain weight functions $w$ on $\partial D$ (if $D$ is bounded, $w\equiv 1$), $d\sigma$ is the area measure of $\partial D$, $\widetilde{h}_{0}$ is related to all the boundary data $h_{j}$, $0\leq j\leq m-1$, $\mathcal{M}_{1}$ is the classical single layer potential operator, and the constant $C$ depends only on $m,n,p$ and $D$.
   \end{description}

 Moreover, as the classical results for the Laplace's equation, in the case of bounded Lipschitz domains, we also have the following estimates of the solutions:
 \begin{itemize}
   \item $\|M(u)\|_{L^{p}(\partial D)}\leq C\sum_{j=0}^{m-1}\|f_{j}\|_{L^{p}(\partial D)}$ for the polyharmonic Dirichlet problem (simply, PHD problem);
   \item $\|M(\nabla u)\|_{L^{p}(\partial D)}\leq C\sum_{j=0}^{m-1}\|g_{j}\|_{L^{p}(\partial D)}$ and $\|u\|_{L^{p}(D)}\leq C\sum_{j=0}^{m-1}\|g_{j}\|_{L^{p}(\partial D)}$ for the polyharmonic Neumann problem (simply, PHN problem);
   \item $\|M(\nabla u)\|_{L^{p}(\partial D)}\leq C\sum_{j=0}^{m-1}\|h_{j}\|_{L_{1}^{p}(\partial D)}$ and $\|u\|_{L^{p}(D)}\leq C\sum_{j=0}^{m-1}\|h_{j}\|_{L_{1}^{p}(\partial D)}$for the polyharmonic regularity problem (simply, PHR problem),
 \end{itemize}
where $M(u)$ and
$M(\nabla u)$ are respectively the non-tangential maximal
functions of $u$ and $\nabla u$, which is
defined by
\begin{equation}
M(F)(Q)=\sup_{X\in \Gamma_{\gamma}(Q)}|F(X)|,\,\,for \,\,Q\in
\partial D,
\end{equation}
where $\Gamma_{\gamma}(Q)$ is the non-tangential approach region,
viz.,
\begin{equation}
\Gamma_{\gamma}(Q)=\{X\in D: |X-Q|<\gamma\, \mathrm{dist}(X,\partial
D)\}
\end{equation}
in which $\gamma>1$. It is worthy to note that the non-tangential
maximal functions $M(F)$, and the non-tangential limits
$\displaystyle\lim_{\substack{X\rightarrow P\\ X\in
\Gamma_{\gamma}(P), P\in\partial D}}\!\!\!\!\!F(X)$ throughout this article,
are defined for all $\gamma>0$, so we always elide the subscript
$\gamma$ in proper places and denote $\Gamma_{\gamma}(\cdot)$ only
by $\Gamma(\cdot)$. It is also clear that all the boundary data in BVPs (1.1)-(1.3) are non-tangential.
Throughout this paper, all the spaces $L_{1}^{p}(\partial D, wd\sigma)$ have the same sense as the case of Laplace equation (for the details, see \cite{dk,dk1,v0}).

Since the late of 1970s, there was a great deal of activity on the
study of boundary value problems for partial
differential equations in Lipschitz domains. The first breakthrough
was due to Dahlberg. In 1977, through a careful analysis of the
Poisson kernel of a Lipschitz domain $D$ with which given, his
showed that there exists an $\varepsilon>0$ depending only on the
geometry of $D$ such that the Dirichlet problem is solvable for the
data in $L^{p}(\partial D, d\sigma)$, $2-\varepsilon<p<\infty$ (see
\cite{dah1,dah2,dah3}). In 1978, Fabes, Jodeit and Riviere used
Calder\'on theorem on the boundedness of the Cauchy integrals on
Lipschitz curves for a special case \cite{cal}, to extend the
classical method of layer potentials to $C^{1}$ domains. Thus they
resolved the Dirichlet and Neumann problems for Laplace's equation,
with $L^{p}(\partial D, d\sigma)$ and optimal estimates, for $C^{1}$
domains \cite{fjv}. In 1979, by using an identity due to Rellich,
Jerison and Kenig gave a simple proof of Dahlberg's results and
resolved the Neumann problem on Lipschitz domains, with
$L^{2}(\partial D, d\sigma)$ and optimal estimates
\cite{jk1,jk2,jk3}. In 1981, Coifman, McIntosh and Meyer established
their deep theorem on the boundedness of the Cauchy integral on any
Lipschitz curve for general case \cite{cmm}. Using
Coifman-McIntosh-Meyer theorem and Rellich type formula, in 1982,
Verchota extended the $C^{1}$ results of Fabes, Jodeit and Riviere
to the Dirichlet problem in $L^{2}(\partial D, d\sigma)$ for
Laplace's equation in Lipschitz domains in terms of the method of
layer potentials \cite{v0}. It was due to Dahlberg and Kenig to
resolve the Neumann problem in $L^{p}(\partial D, d\sigma)$ for
Laplace's equation in Lipschitz domains in 1987 \cite{dk1}.
Thereafter, the technique of layer potentials became an
overwhelming method in the study of BVPs in $C^{1}$ and Lipschitz
domains of Euclidean spaces or Riemann manifolds, with various
boundary data, including the H\"older continuous, $L^{p}$, Hardy,
Besov, Sobolev types etc.. The BVP types included Dirichlet, Neumann,
Robin and mixed problems for elliptic equations and elliptic systems
\cite{d,dk,dk1,dkpv,dkv,lcb,ls,mr,mm,mt,mt1,mt2,ob,pv,pv1,pv2,pv3,v,v0,v1,v2,s,imm}.
Although there were some works for higher order equations
(principally, polyharmonic \cite{dkpv,imm,pv3,v2}), however, the most
were second order elliptic boundary value problems \cite{mr,k} and
biharmonic boundary value problems \cite{dkv,mm,pv,pv1,pv2,v1,v}.

In this paper, we introduce higher order conjugate Poisson and
Poisson kernels, which are higher order analogues of the
classical conjugate Poisson and Poisson kernels, as well as the polyharmonic fundamental solutions, and define
multi-layer potentials in terms of Poisson field and the polyharmonic fundamental solutions, in which the former formed by the
higher order conjugate Poisson and Poisson kernels. Then by the
multi-layer potentials, we solve three classes of boundary value
problems (i.e., Dirichlet, Neumann and regularity problems) with $L^{p}$
boundary data for polyharmonic equations in Lipschitz domains and
give integral representation (or potential) solutions of these problems. That is, combining with the known second order
results of Dahlberg, Kenig and Verchota etc., we
resolve the higher order elliptic boundary value problems
(1.1)-(1.3) in Lipschitz domains.

\section{Higher order conjugate Poisson and
Poisson kernels}

It is well-known that the conjugate Poisson and Poisson kernels in
$\mathbb{R}^{n+1}$ can be unifiedly denoted as the following form up to a different constant
(see \cite{sw})
\begin{equation}
\mathcal{P}_{j}(x)=C_{n}\frac{x_{j}}{|x|^{n+1}},
\end{equation}
where $x=(x_{1}, x_{2},\ldots, x_{n+1})\in \mathbb{R}^{n+1}$, $1\leq
j\leq n+1$ and
\begin{equation}C_{n}=\frac{1}{\omega_{n}}=\frac{\Gamma(\frac{n+1}{2})}{2\pi^{\frac{n+1}{2}}},\end{equation}
in which $\omega_{n}$ is the surface area of the unit sphere $S^{n}$
in $\mathbb{R}^{n+1}$.

In what follows, we will introduce higher order conjugate Poisson and
Poisson kernels in terms of $\mathcal{P}_{j}$.

\begin{lem}
Let $x=(x_{1}, x_{2},\ldots, x_{n+1})\in\mathbb{R}^{n+1}$, then for
any $s\in \mathbb{R}$ and $1\leq j\leq n+1$,
\begin{equation}
\Delta \left(x_{j}|x|^{s}\right)=s(s+n+1)x_{j}|x|^{s-2}
\end{equation}
and
\begin{equation}
\Delta
\left(x_{j}|x|^{s}\log|x|\right)=s(s+n+1)x_{j}|x|^{s-2}\log|x|+(2s+n+1)x_{j}|x|^{s-2},
\end{equation}
where $\Delta=\sum_{k=1}^{n+1}\frac{\partial^{2}}{\partial
x_{k}^{2}}$ and $|x|=\sqrt{x_{1}^{2}+\cdots+x_{n+1}^{2}}$.
\end{lem}
\begin{proof} It is the same as in \cite{dqw2}.\end{proof}

Denote that
\begin{equation}\alpha_{s}=s(s+n+1)\end{equation}
for any $s\in \mathbb{R}$. Thus, when $s\neq0$, we can rewrite (2.3)
and (2.4) as follows:
\begin{equation}
\Delta \left(\frac{1}{\alpha_{s}}x_{j}|x|^{s}\right)=x_{j}|x|^{s-2}
\end{equation}
and
\begin{equation}
\Delta
\left(\frac{1}{\alpha_{s}}x_{j}|x|^{s}\log|x|\right)=x_{j}|x|^{s-2}\log|x|+\left(\frac{1}{s}+\frac{1}{s+n+1}\right)x_{j}|x|^{s-2}.
\end{equation}
As a convention, we take that $\alpha_{0}=1$. Moreover, we also have
\begin{equation}
\Delta \left(\frac{1}{n+1}x_{j}\log|x|\right)=x_{j}|x|^{-2}.
\end{equation}

\begin{lem}
Suppose that $x=(x_{1}, x_{2},\ldots, x_{n+1}), v=(v_{1},
v_{2},\ldots, v_{n+1})\in \mathbb{R}^{n+1}$. Let
\begin{equation}
D_{1}^{(j)}(x,v)=-\mathcal{P}_{j}(x-v).
\end{equation}
For $m\in \mathbb{N}$ and $m\geq2$, define
\begin{align}
D_{m}^{(j)}(x,v)=\frac{c_{n}}{\beta_{1}\beta_{2}\cdots\beta_{m-1}}(x_{j}-v_{j})|x-v|^{2m-(n+3)}
\end{align}
if  $n$ is even, and
\begin{align}
D_{m}^{(j)}(x,v)=\begin{cases}\frac{c_{n}}{\beta_{1}\beta_{2}\cdots\beta_{m-1}}(x_{j}-v_{j})|x-v|^{2m-(n+3)},\,\,\,\,m\leq \frac{n+1}{2},\vspace{2mm}\\
\frac{c_{n}}{(n+1)\beta_{1}\beta_{2}\cdots\beta_{\frac{n+1}{2}-1}\alpha_{2}\alpha_{4}\cdots\alpha_{2m-n-3}}(x_{j}-v_{j})|x-v|^{2m-(n+3)}\vspace{1mm}\\
\times\left[\log|x-v|-\sum_{t=1}^{m-\frac{n+3}{2}}\left(\frac{1}{2t}+\frac{1}{2t+n+1}\right)\right],\,\,\,m\geq\frac{n+3}{2}
\end{cases}
\end{align}
if $n$ is odd, where
\begin{equation}\beta_{k}=\alpha_{2k-n-1},\,\,\,k=1,2,\ldots,m-1,\end{equation}
$\alpha_{s}$ is given by (2.5) and $c_{n}=-C_{n}$, $C_{n}$ is given by (2.2). Then
\begin{equation}
\Delta D_{1}^{(j)}(x,v)=0\,\,\mbox{and}\,\,\Delta
D_{m}^{(j)}(x,v)=D_{m-1}^{(j)}(x,v), \,\,m\geq2.
\end{equation}
\end{lem}

\begin{proof}
By direct calculations, it immediately follows from (2.6)-(2.8).
\end{proof}

In the following, we need to introduce ultraspherical polynomials
\cite{aar,sz}, $P_{l}^{(\lambda)}$ and $Q_{l}^{(\lambda)}$, which
can be respectively defined by the generating functions

\begin{equation}
(1-2r\xi+r^{2})^{-\lambda}=\sum_{l=0}^{\infty}P_{l}^{(\lambda)}(\xi)r^{l}
\end{equation}
and
\begin{equation}
(1-2r\xi+r^{2})^{-\lambda}\log(1-2r\xi+r^{2})=\sum_{l=0}^{\infty}Q_{l}^{(\lambda)}(\xi)r^{l},
\end{equation}
where $\lambda\neq0$, $0\leq|r|<1$ and $|\xi|\leq1$.
$P_{l}^{(\lambda)}$ and $Q_{l}^{(\lambda)}$ have the following
explicit expressions:
\begin{align}
P_l^{(\lambda)}(\xi)=&\frac{1}{l!}\left\{\frac{d^{l}}{dr^{l}}\left[(1-2r\xi+r^{2})^{-\lambda}\right]\right\}_{r=0}\\
=&\sum_{j=0}^{[\frac{l}{2}]}(-1)^j\frac{\Gamma(l-j+\lambda)}{\Gamma(\lambda)j!(l-2j)!}(2\xi)^{l-2j}\nonumber
\end{align}
and
\begin{align}
Q_l^{(\lambda)}(\xi)=&-\frac{d}{d \lambda}\left[P_l^{(\lambda)}(\xi)\right]\\
=&\sum_{j=0}^{[\frac{l}{2}]}\sum_{k=0}^{l-j-1}(-1)^{j+1}\frac{\Gamma(l-j+\lambda)}{(\lambda+k)\Gamma(\lambda)j!(l-2j)!}(2\xi)^{l-2j},\nonumber
\end{align}
where $[\frac{l}{2}]$ denotes the integer part of $\frac{l}{2}$. If
necessary, for some special values of $\lambda$, say
$\lambda=\lambda_{0}$, the above expressions may be extended and
interpreted as limits for $\lambda\rightarrow\lambda_{0}$ (for
example, $\lambda$ is a non-positive integer). Some other properties
of the ultraspherical polynomials can be also found in
\cite{aar,sz}.

For sufficiently large $|v|\geq |x|$ and any real numbers
$\lambda\neq0$,
\begin{align}
|x-v|^{-2\lambda}&=(|v|^{2}-2x\cdot
v+|x|^{2})^{-\lambda}\\
&=|v|^{-2\lambda}\left[1-2\frac{|x|}{|v|}\left(\frac{x}{|x|}\cdot
\frac{v}{|v|}\right)+\frac{|x|^{2}}{|v|^{2}}\right]^{-\lambda}\nonumber\\
&=|v|^{-2\lambda}\sum_{l=0}^{\infty}P_{l}^{(\lambda)}(x_{S^{n}}\cdot
v_{S^{n}})\left(\frac{|x|}{|v|}\right)^{l}\nonumber\\
&=\sum_{l=0}^{\infty}|x|^{l}P_{l}^{(\lambda)}(x_{S^{n}}\cdot
v_{S^{n}})|v|^{-(l+2\lambda)},\nonumber
\end{align}
where $x_{S^{n}}=\frac{x}{|x|}$ and $v_{S^{n}}=\frac{v}{|v|}$. Obviously, $x_{S^{n}}, v_{S^{n}}\in S^{n}$.

Similarly, we have
\begin{align}
&|x-v|^{-2\lambda}\log|x-v|\\
=&|x-v|^{-2\lambda}\left[\frac{1}{2}\log\frac{|x-v|^{2}}{|v|^{2}}+\log|v|\right]\nonumber\\
=&(|v|^{2}-2x\cdot v+|x|^{2})^{-\lambda}\left[\frac{1}{2}\log\frac{|v|^{2}-2x\cdot
v+|x|^{2}}{|v|^{2}}+\log|v|\right]\nonumber\\
=&|v|^{-2\lambda}\left[1-2\frac{|x|}{|v|}\left(\frac{x}{|x|}\cdot
\frac{v}{|v|}\right)+\frac{|x|^{2}}{|v|^{2}}\right]^{-\lambda}\Big\{\frac{1}{2}\log\left[1-2\frac{|x|}{|v|}\left(\frac{x}{|x|}\cdot
\frac{v}{|v|}\right)+\frac{|x|^{2}}{|v|^{2}}\right]\nonumber\\
&+\log|v|\Big\}\nonumber\\
=&\frac{1}{2}|v|^{-2\lambda}\sum_{l=0}^{\infty}Q_{l}^{(\lambda)}(x_{S^{n}}\cdot
v_{S^{n}})\left(\frac{|x|}{|v|}\right)^{l}+|v|^{-2\lambda}\log|v|\sum_{l=0}^{\infty}P_{l}^{(\lambda)}(x_{S^{n}}\cdot
v_{S^{n}})\left(\frac{|x|}{|v|}\right)^{l}\nonumber\\
=&\frac{1}{2}\sum_{l=0}^{\infty}|x|^{l}Q_{l}^{(\lambda)}(x_{S^{n}}\cdot
v_{S^{n}})|v|^{-(l+2\lambda)}+\sum_{l=0}^{\infty}|x|^{l}\log|v|P_{l}^{(\lambda)}(x_{S^{n}}\cdot
v_{S^{n}})|v|^{-(l+2\lambda)}.\nonumber
\end{align}

\begin{defn}
Let $f$ be a continuous function defined in $\mathbb{R}^{n+1}$ that
can be expanded as
\begin{equation}
f(\zeta)=\sum_{k=-\infty}^{m}c_{k}(\zeta)|\zeta|^{k}
\end{equation}
for sufficiently large $|\zeta|$, where the integer $m\geq -(n+1)$ and the
coefficient functions $c_{k}(\zeta)$ are continuous in
$\mathbb{R}^{n+1}$. Denote
\begin{equation}
\mathrm{S. P.}
[f](\zeta)=\sum_{k=0}^{m}c_{k}(\zeta)|\zeta|^{k}+\sum_{k=1}^{n+1}
c_{-k}(\zeta)\frac{1}{|\zeta|^{k}}
\end{equation}
and
\begin{equation}
\mathrm{I. P.}
[f](\zeta)=\sum_{k=n+2}^{\infty}c_{-k}(\zeta)\frac{1}{|\zeta|^{k}}
\end{equation}
for sufficiently large $|\zeta|$. If $\mathrm{I. P.} [f]$ is $L^{p}$ integrable in the complement of a sufficiently large ball centered at the origin in $\mathbb{R}^{n+1}$ for $p\geq1$, then $\mathrm{S. P.} [f]$ is
called the singular part of $f$ and $\mathrm{I. P.} [f]$ is called
the integrable part of $f$ at infinity in the $L^{p}$ sense, $p\geq1$.
\end{defn}

We immediately have

\begin{prop}
Let $f$ be defined as in Definition 2.3, then for sufficiently large
$|\zeta|$,
\begin{equation}
f(\zeta)=\mathrm{S. P.} [f](\zeta)+\mathrm{I. P.} [f](\zeta).
\end{equation}
\end{prop}

\begin{defn}
Let
\begin{equation}
K_{m}^{(j)}(x,v)=
D_{m}^{(j)}(x,v)-\mathrm{S. P.} [D_{m}^{(j)}](x,v)\,\,\,\,for\,\, x\neq v,
\end{equation}
where
\begin{align}
\mathrm{S. P.}
[D_{m}^{(j)}](x,v)=&\frac{c_{n}}{\beta_{1}\beta_{2}\cdots\beta_{m-1}}(x_{j}-v_{j})\Big[\sum_{l=0}^{2m-1}P_{l}^{(\frac{n+3}{2}-m)}(x_{S^{n}}\cdot
v_{S^{n}})\\
&\times\min\left(\left|\frac{x}{v}\right|^{l}, \left|\frac{x}{v}\right|^{-l}\right)\times\max\left(|x|^{2m-n-3}, |v|^{2m-n-3}\right)\Big]\nonumber
\end{align}
for any $m$ and even $n$, or any odd $n$ with $m\leq\frac{n+1}{2};$
and
\begin{align}
\mathrm{S. P.} [D_{m}^{(j)}](x,v)=&\frac{c_{n}}{(n+1)\beta_{1}\beta_{2}\cdots\beta_{\frac{n+1}{2}-1}\alpha_{2}\alpha_{4}\cdots\alpha_{2m-1}}(x_{j}-v_{j})\\
&\times\Big\{\frac{1}{2}\Big[\sum_{l=0}^{2m-1}Q_{l}^{(\frac{n+3}{2}-m)}(x_{S^{n}}\cdot
v_{S^{n}})\nonumber\\
&\times\min\left(\left|\frac{x}{v}\right|^{l}, \left|\frac{x}{v}\right|^{-l}\right)\times\max\left(|x|^{2m-n-3}, |v|^{2m-n-3}\right)\Big]\nonumber\\
&+\left[\log(\max(|x|, |v|))-\sum_{t=1}^{m-\frac{n+3}{2}}\left(\frac{1}{2t}+\frac{1}{2t+n+1}\right)\right]\nonumber\\
&\times\Big[\sum_{l=0}^{2m-1}P_{l}^{(\frac{n+3}{2}-m)}(x_{S^{n}}\cdot
v_{S^{n}})\times\min\left(\left|\frac{x}{v}\right|^{l}, \left|\frac{x}{v}\right|^{-l}\right)\nonumber\\
&\times\max\left(|x|^{2m-n-3}, |v|^{2m-n-3}\right)\Big]\Big\}\nonumber
\end{align}
for any odd $n$ with $m\geq\frac{n+3}{2}$, in which
$\alpha_{s},\,\beta_{s}$ and $c_{n}$ are given as in Lemma 2.2, and the ultraspherical polynomials
$P_{l}^{(\frac{n+3}{2}-m)},\,Q_{l}^{(\frac{n+3}{2}-m)}$ are defined by
(2.16) and (2.17). Then $K_{m}^{(j)}(x,v)$, $1\leq j\leq n+1$, are
said to be the $m$th order conjugate Poisson and Poisson kernels.
\end{defn}

By the above definition, we immediately obtain that
\begin{prop}
\begin{equation}K_{m}^{(j)}(x,v)=-K_{m}^{(j)}(v,x)\end{equation}
with $x\neq v$ for any $m\in \mathbb{N}$ and $1\leq j\leq n+1$.
\end{prop}

\begin{rem}
Let $x=(x_{1}, x_{2}, \ldots, x_{n}, y)\in \mathbb{R}^{n+1}_{+}$ and
$v=(\underline{v}, 0)$ with $\underline{v}=(v_{1}, v_{2}, \ldots,
v_{n})$, then $2K_{m}^{(n+1)}(x,v)$ are just the higher order
Poisson kernels with a different singular part, $G_{m}(x,\underline{v})$, introduced in
\cite{dqw2}. Using those kernels, we have resolved the following
polyharmonic Dirichlet problems with $L^p$ data in the upper-half
space, $\mathbb{R}_{+}^{n+1}$
\begin{equation}
\begin{cases}
\Delta^{m}u=0 \,\,\mathrm{in} \,\,\mathbb{R}_{+}^{n+1} \vspace{1mm}\\
\Delta^{j}u=f_{j} \,\,\mathrm{on}
\,\,\partial\mathbb{R}_{+}^{n+1}=\mathbb{R}^{n},
\end{cases}
\end{equation}
where $n\geq 2$,
$\mathbb{R}_{+}^{n+1}=\mathbb{R}^{n}\times
\mathbb{R}_{+}=\{x=(\underline{x},y): \underline{x}\in
\mathbb{R}^{n},\,y\in \mathbb{R},\, y>0\}$,
$\underline{x}=(x_{1},\ldots,x_{n})$,
$\Delta\equiv\Delta_{n+1}:=\sum_{k=1}^{n}\frac{\partial^{2}}{\partial
x_{k}^{2}}+\frac{\partial^{2}}{\partial y^{2}}$, $f_{j}\in
L^{p}(\mathbb{R}^{n})$, $m\in \mathbb{N}$, $0\leq j< m$, and $p\geq
1$.
\end{rem}

\section{Multi-layer $\mathcal{D}$-potentials}
With the aforementioned preliminaries,  in the present section, we
introduce one class of multi-layer potentials in terms of the higher order
conjugate Poisson and Poisson kernels, which are higher order analogues of the classical double layer potential.

Let $X=(x_{1}, x_{2}, \ldots, x_{n+1})$, $Y=(y_{1}, y_{2}, \ldots,
 y_{n+1})\in \mathbb{R}^{n+1}$ and $X\neq Y$, for any natural number $m\geq
1$, define
\begin{equation}
K_{m}(X, Y)=(K_{m}^{(1)}(X,Y), K_{m}^{(2)}(X,Y), \ldots,
K_{m}^{(n+1)}(X,Y)),
\end{equation}
where $K_{m}^{(j)}$, $1\leq j\leq n+1$, are the $m$th order conjugate Poisson
and Poisson kernels. $K_{m}$ is called the $m$th order
Poisson field.

\begin{defn}
Let $D$ be a simply connected (bounded or unbounded) domain in
$\mathbb{R}^{n+1}$ with the boundary $\partial D$ and $k\in
\mathbb{N}\cup\{\infty\}$, $C^{k}(D)$ denotes the set of the
functions that have continuous partial derivatives of  order $k$ in
$D$. If $f$ is a continuous function defined on $D\times\partial D$
satisfying $f(\cdot,v)\in C^{k}(D)$ for any fixed $v\in
\partial D$ and $f(x,\cdot)\in C(\partial D)$ for any fixed $x\in
D$, then $f$ is said to be $C^{k}\times C$ on $D\times\partial D$
and written as $f\in (C^{k}\times C)(D\times\partial D)$. When $f$
is vector-valued, $f\in (C^{k}\times C)(D\times\partial D)$ means
that all of its components are in $(C^{k}\times C)(D\times\partial
D)$.
\end{defn}

\begin{defn}Let $D$ be a Lipschitz domain in $\mathbb{R}^{n+1}$, with the
boundary $\partial D$. Set
\begin{equation}
M_{j}f(X)=\int_{\partial D}\langle K_{j}(X,Q), n_{Q}\rangle
f(Q)d\sigma(Q), \,\,X\in D,
\end{equation}
where $1\leq j< \infty$, $K_{j}$ is the $j$th order Poisson field,
$n_{Q}$ is the outward unit normal at $Q\in \partial D$, $\langle
\cdot, \cdot\rangle$ is the inner product in
$\ell^{2}(\mathbb{R}^{n+1})$, $d\sigma$ is the surface measure on
$\partial D$, and $f\in L^{p}(\partial D)$ for some suitable $p$.
$M_{j}f$ is called the $j$th-layer $\mathcal{D}$-potential of $f$.
\end{defn}

\begin{rem}
By the above definition, $M_{1}f$ is the classical double layer
potential.
\end{rem}
Define
\begin{equation}
Tf(P)=\lim_{\epsilon\rightarrow0}\int_{\partial D\setminus
B_{\epsilon}(P) }\langle K_{1}(P,Q), n_{Q}\rangle f(Q)d\sigma(Q),
\,\,P\in \partial D,
\end{equation}
where $B_{\epsilon}(P)=\{Q\in \mathbb{R}^{n+1}: |Q-P|<\epsilon\}$. Hence the adjoint operator of $T$ is given by
\begin{equation}
T^{*}f(P)=\lim_{\epsilon\rightarrow0}\int_{\partial D\setminus
B_{\epsilon}(P) }\langle K_{1}(Q,P), n_{P}\rangle f(Q)d\sigma(Q),
\,\,P\in \partial D.
\end{equation}

Due to Dahlberg, Kenig and Verchota et al., we have

\begin{lem}[\cite{dk1,v0}]
There exists $\varepsilon=\varepsilon(D)>0$ such that
$\pm\frac{1}{2}I-T$ is invertible in $L^{p}(\partial D)$,
$2-\varepsilon<p<\infty$, and $\pm\frac{1}{2}I-T^{*}$ is invertible
in $L^{p}(\partial D)$, $1<p<2+\varepsilon$.
\end{lem}

By the properties of higher order conjugate Poisson and
Poisson kernels, we have
\begin{thm}
Let $\{K_{m}\}_{m=1}^{\infty}$ be the sequence of the Poisson
fields, and $D$ be a Lipschitz graph domain in $\mathbb{R}^{n+1}$,
i.e.,
\begin{equation}
D=\{(\underline{x},x_{n+1})\in \mathbb{R}^{n+1}:
x_{n+1}>\varphi(\underline{x}), \,\underline{x}=(x_{1}, x_{2},
\ldots, x_{n})\in \mathbb{R}^{n}\},
\end{equation}
where $\varphi: \mathbb{R}^{n}\rightarrow \mathbb{R}$ is Lipschitz
continuous; namely,
$|\varphi(\underline{x})-\varphi(\underline{x}^{\prime})|\leq L
|\underline{x}-\underline{x}^{\prime}|$ with $0<L<\infty$, and set $\varphi(0)>0$, then
\begin{enumerate}
  \item For all $m\in \mathbb{N}$, $K_{m}\in (C^{\infty}\times
        C)(D\times \partial D)$, the non-tangential
        boundary value \begin{equation*}
        \lim_{\substack {X\rightarrow P \\X\in \Gamma(P),\, Q\in
        \partial D}}K_{m}(X,Q)=K_{m}(P,Q)
        \end{equation*}
        exists for all $P\in \partial D$ and $ P\ne Q\in
        \partial D$; $K_{m}(\cdot, P)$ can be continuously extended to
        $\overline{D}\setminus\{P\}$ for any fixed $P\in
        \partial D$;
  \item For  $m\geq 2$,
        \begin{equation*}
        |K_{m}(X,Q)|\leq M\frac{|X-Q|}{\left(1+|Q|^{2}\right)^{\frac{n+2+\epsilon}{2}}}
        \end{equation*}
        for any $(X,Q)\in D_{c}\times\{Q\in\partial D: |Q|>T\}$, where $0<\epsilon<1$,
        $D_{c}$ is any compact subset of $\overline{D}$,
        $T$ is a sufficiently large positive real number and $M$ denotes
        some positive constant depending only on $\epsilon, D_{c}$ and $T$;
  \item $\Delta_{X} K_{1}(X, Y)=-\Delta_{Y} K_{1}(X, Y)=0$ and $\Delta_{X} K_{m}(X, Y)=-\Delta_{Y} K_{m}(X, Y)=K_{m-1}(X, Y)$ for any $m>1$, $X, Y\in \mathbb{R}^{n+1}\setminus\{0\}$ and $X\neq Y$, where $\Delta_{X}=\sum_{j=1}^{n+1}\frac{\partial}{\partial x^{j}}$ and $\Delta_{Y}=\sum_{j=1}^{n+1}\frac{\partial}{\partial y^{j}}$;
  \item The non-tangential limit \begin{equation}\lim_{\substack{X\rightarrow P\\ X\in \Gamma(P)}}\int_{\partial D}\langle K_{1}(X, Q), n_{Q}\rangle f(Q)d\sigma(Q)
        =\frac{1}{2}f(P)+Tf(P),\end{equation}
        for any $f\in L^{p}(\partial D)$, $1\leq p<\infty$;
  \item The non-tangential limit \begin{equation}\lim_{\substack{X\rightarrow P\\ X\in \Gamma(P)}}\int_{\partial D}\langle K_{m}(X, Q), n_{Q}\rangle f(Q)d\sigma(Q)
        =\mathrm{K}_{m}f(P)\end{equation}
        for any $m\geq2$ and $f\in L^{p}(\partial D)$, $1\leq p\leq\infty$,
        where \begin{equation}\mathrm{K}_{m}f(P)=\int_{\partial D}\langle K_{m}(P, Q), n_{Q}\rangle f(Q)d\sigma(Q),\,\,P\in\partial D\end{equation}\end{enumerate}
        which is a principle value integral defined as (3.3).
\end{thm}

\begin{rem}
In this theorem and what follows, with respect to the Lipschitz graph domains, we emphasize that the Lipschitz funtion $\varphi$ should satisfy the condition $\varphi(0)>0$ in order to avoid $0\in \overline{D}$. This is only a technical requirement to guarantee the $L^{p}$-integrability on $\partial D$ and continuity on $\overline{D}$ of the kernels $K_{m}^{(j)}$. If $0\in \overline{D}$, we can take any fixed point $x_{0}\in \mathbb{R}^{n+1}\setminus\overline{D}$ and use it to redefine the singular parts of $K^{(j)}_{m}$ in (2.25) and (2.26) with the terms $|x|$ and $|v|$ replaced respectively by $|x-x_{0}|$ and $|v-x_{0}|$. So we do, then the above theorem and main results in the paper still hold with $x_{0}$ in place of $0$.
\end{rem}

\begin{proof}
By using the definition of the singular part, $\mathrm{S.
P.}[\cdot]$, and performing
similar calculations as to get (2.18) and (2.19), we get (2.25)
and (2.26). Note the explicit expressions (2.25) and (2.26), it
immediately follows that for any $m\in \mathbb{N}$, $K_{m}\in
(C^{\infty}\times C)(D\times \partial D)$, the non-tangential
boundary value
\begin{equation*}
        \lim_{\substack {X\rightarrow P \\X\in \Gamma(P),\, Q\in
        \partial D}}K_{m}(X,Q)=K_{m}(P,Q)
        \end{equation*}
        exists for all $P\in \partial D$ and $ P\ne Q\in
        \partial D$. Furthermore, $K_{m}(\cdot,P)$ can be continuously extended to
        $\overline{D}\setminus\{P\}$ for any fixed $P\in
        \partial D$, i.e., the claim (1) holds.

Note that
\begin{align*}
D_{1}^{(j)}(x,v)=-\frac{1}{\omega_{n}}\mathcal{P}_{j}(x-v)=-\frac{1}{\omega_{n}}\frac{x_{j}-v_{j}}{|x-v|^{n+1}}.
\end{align*}
So by the definition of the singular part,
\begin{equation}
\mathrm{S. P.} [D_{1}^{(j)}](x,v)\equiv0.
\end{equation}
Therefore
\begin{equation}
\langle K_{1}(X,Q), n_{Q}\rangle=\frac{1}{\omega_{n}}\frac{\langle
Q-X, n_{Q}\rangle}{|X-Q|^{n+1}}.
\end{equation}
Then by the theory of classical layer potentials \cite{f1,v0},
\begin{equation*}\lim_{\substack{X\rightarrow P\\ X\in
\Gamma(P)}}\int_{\partial D}\langle K_{1}(X, Q), n_{Q}\rangle
f(Q)d\sigma(Q)
        =\frac{1}{2}f(P)+Tf(P),\end{equation*}
        for any $f\in L^{p}(\partial D)$, $1\leq p<\infty$. Moreover, by the definition and Taylor's expansion, for sufficiently large $|v|>|x|,$
\begin{align}
\mathrm{I.P.}[D_{m}^{(j)}](x,v)=\begin{cases}
A_{m,n}(x_{j}-v_{j})C_{m,n}(x,v)\frac{1}{|v|^{n+3}},\,\, n\,even\,and\,any \,m, or\, n\, odd \,and\, m\leq\frac{n+1}{2},\vspace{2mm}\\
B_{m,n}(x_{j}-v_{j})\left[\widetilde{C}_{m,n}(x,v)+\widehat{C}_{m,n}(x,v)\log|v|\right]\frac{1}{|v|^{n+3}},\,\,n\,odd
\,and\, m\geq\frac{n+3}{2},
\end{cases}
\end{align}
where $A_{m,n}$ and $ B_{m,n}$ are positive constants depending only
on $m$ and $n$,
\begin{align}C_{m,n}(x,v)=|x|^{2m}\left\{\frac{d^{2m}}{dr^{2m}}\left[(1-2r(x_{S^{n}}\cdot
v_{S^{n}})+r^{2})^{m-\frac{n+3}{2}}\right]\right\}_{r=\theta}\end{align}
and
\begin{align}\widetilde{C}_{m,n}(x,v)=&|x|^{2m}\Big\{\frac{d^{2m}}{dr^{2m}}\Big[(1-2r(x_{S^{n}}\cdot
v_{S^{n}})+r^{2})^{m-\frac{n+3}{2}}\\
&\times\left[\frac{1}{2}\log(1-2r(x_{S^{n}}\cdot
v_{S^{n}})+r^{2})-\sum_{t=1}^{m-\frac{n+3}{2}}\left(\frac{1}{2t}+\frac{1}{2t+n+1}\right)\right]\Big]\Big\}_{r=\vartheta}\nonumber\end{align}
as well as
\begin{align}\widehat{C}_{m,n}(x,v)=|x|^{2m}\left\{\frac{d^{2m}}{dr^{2m}}\left[(1-2r(x_{S^{n}}\cdot
v_{S^{n}})+r^{2})^{m-\frac{n+3}{2}}\right]\right\}_{r=\varrho}\end{align}
with $0<\theta,\vartheta,\varrho<\frac{|x|}{|v|}<1$. Note that
\begin{equation}\lim_{|v|\rightarrow \infty}\frac{\log|v|}{|v|^{\epsilon}}=0\end{equation} for any $\epsilon>0$. Therefore, for any compact subset
$D_{c}$ of $\overline{D}$  and $X\in D_{c},$  by the continuity of
$C_{m,n}$, $\widetilde{C}_{m,n} $ and $\widehat{C}_{m,n}$, we have
\begin{equation}
|K_{m}^{(j)}(X,Q)|=\left|\mathrm{I.P.}[D_{m}^{(j)}](X,Q)\right|\leq
M\frac{|x_{j}-v_{j}|}{\left(1+|Q|^{2}\right)^{\frac{n+2+\epsilon}{2}}},
\end{equation}
where $0<\epsilon<1$, $(X,Q)\in D_{c}\times\{Q\in\partial D: |Q|>T\},$ $T$ is a
sufficiently large positive real number and $M$ is a positive
constant depending only on $\epsilon$, $D_{c}$ and $T$. Thus the claims (2) and
(4) are established.

From (2.25) and (2.26), we can simply denote
\begin{align}
\mathrm{S. P.}
[D_{m}^{(j)}](x,v)=&C_{m}(x_{j}-v_{j})\sum_{l=0}^{2m-1}c_{m,l}(x,v)|v|^{2m-n-3-l},
\end{align}
where $C_{m}$ is a constant depending only on $m,n$, and the
coefficient functions $c_{m,l}$ can be explicitly
expressed by the ultraspherical polynomials $P_{l}^{(\frac{n+3}{2}-m)}(x_{S^{n}}\cdot v_{S^{n}})$,
$Q_{l}^{(\frac{n+3}{2}-m)}(x_{S^{n}}\cdot v_{S^{n}})$, $|x|^{l}$ and $\log|v|$. Therefore,
\begin{align}
\Delta\Big[\mathrm{S. P.}
[D_{m}^{(j)}](x,v)\Big]=&C_{m}\sum_{l=0}^{2m-1}\Delta[(x_{j}-v_{j})c_{m,l}(x,v)]|v|^{2m-n-3-l}.
\end{align}
By Lemma 2.2, we have
\begin{equation}
\Delta K_{m}^{(j)}-K_{m-1}^{(j)}=\mathrm{S. P.}
[D_{m-1}^{(j)}]-\Delta\Big[\mathrm{S. P.} [D_{m}^{(j)}]\Big]
\end{equation}
for any $m\geq 2$. Due to (3.16) and (3.17), for sufficiently large $v$ (in deed, for all $v$),
\begin{equation*}
\Delta K_{m}^{(j)}=K_{m-1}^{(j)}\,\,\,\text{and}\,\,\,\mathrm{S. P.}
[D_{m-1}^{(j)}]=\Delta\Big[\mathrm{S. P.} [D_{m}^{(j)}]\Big]
\end{equation*}
for any $m\geq 2$. By taking into account $\Delta K_{1}=0$, and by Proposition 2.6, the
claim (3) follows.

Finally, we show that the claim (5) holds.\\
\textbf{Case 1}: $2\leq m\leq\frac{n+1}{2}$. Take a splitting,
\begin{align}
\int_{\partial D}\langle K_{m}(X,Q), n_{Q}\rangle f(Q)d\sigma(Q)&=\int_{\partial D\cap B_{\delta}(P)}\langle K_{m}(X,Q), n_{Q}\rangle f(Q)d\sigma(Q)\\
&\,\,\,\,\,\,\,+\int_{\partial D\cap B_{T}(P)\setminus
B_{\delta}(P)}\langle
K_{m}(X,Q), n_{Q}\rangle f(Q)d\sigma(Q)\nonumber\\
&\,\,\,\,\,\,\,+\int_{\partial D\setminus B_{T}(P)}\langle
K_{m}(X,Q), n_{Q}\rangle f(Q)d\sigma(Q)\nonumber\\
&\triangleq\mathrm{I}+\mathrm{II}+\mathrm{III},\nonumber
\end{align}
where $P$ is any fixed point in $\partial D$, $\delta, T>0$,
$\delta$ is sufficiently small while $T$ is sufficiently large,
$X\in \Gamma_{\gamma,\eta}(P)=\{X\in \Gamma_{\gamma}(P):
\mathrm{dist}(X,
\partial D)\leq \eta\}$, $0<\eta<\min\{\delta,\frac{1}{2}\}$, and $f\in
L^{p}(\partial D)$, $1\leq p\leq\infty$. By the claim (1),
$K_{m}^{(j)}(X,Q)$ is continuous on the compact
set $\Gamma_{\gamma,\eta}(P)\times\{Q\in
\partial D:\,\delta\leq|Q-P|\leq T\}$.  Therefore,
\begin{align}
\mathrm{II}\rightarrow \int_{\partial D\cap B_{T}(P)\setminus
B_{\delta}(P)}\langle K_{m}(P,Q), n_{Q}\rangle f(Q)d\sigma(Q)
\,\,\mathrm{as} \, \,X\rightarrow P, \,\,X\in
\Gamma_{\gamma,\eta}(P).
\end{align}
By the claim (2), for sufficiently large $T$ and some fixed $0<\epsilon_{0}<1$,  $X\in
\Gamma_{\gamma,\eta}(P)$  and  $|Q-P|>T$, we have
\begin{equation*}
|K_{m}^{(j)}(X,Q)|\leq
M\frac{|x_{j}-v_{j}|}{(1+|Q|^{2})^{\frac{n+2+\epsilon_{0}}{2}}},
\end{equation*}
where $M$ is a constant depending only on $\delta$, $T$ and $\epsilon_{0}$. So
\begin{equation}
|\langle K_{m}(X,Q), n_{Q}\rangle f(Q)|\leq
M\frac{|X-Q|}{(1+|Q|^{2})^{\frac{n+2+\epsilon_{0}}{2}}}|f(Q)|.
\end{equation}
The RHS of the above inequality belongs to $L^{1}(\partial D),$
because $\frac{|X-Q|}{(1+|Q|^{2})^{\frac{n+2+\epsilon_{0}}{2}}}\in L^{q}(\partial
D)\cap C_{0}(\partial D)$ and $f\in L^{p}(\partial D)$ for any $1\leq p\leq\infty$
and $q\geq1$, where $X\in \Gamma_{\gamma,\eta}(P)$ and $C_{0}(\partial
D)$ is the set of all continuous functions defined on $\partial D$ vanishing at
infinity. Since by (3.22),
$$\langle K_{m}(X,Q), n_{Q}\rangle f(Q)\rightarrow \langle K_{m}(P,Q), n_{Q}\rangle f(Q)$$ as $X\rightarrow P$ for any
$X\in\Gamma_{\gamma,\eta}(P)$ and $|Q-P|>T$, and Lebesgue's
dominated convergence theorem,
\begin{align}
\mathrm{III}\rightarrow \int_{\partial D\setminus B_{T}(P)}\langle
K_{m}(P,Q), n_{Q}\rangle f(Q)d\sigma(Q) \,\,\mathrm{as} \,
\,\,X\rightarrow P, X\in \Gamma_{\gamma,\eta}(P).
\end{align}
Write that
\begin{align}
\mathrm{I}^{(j)}&=\int_{\partial D\cap
B_{\delta}(P)}D_{m}^{(j)}(X,Q)n_{Q}^{(j)}f(Q)d\sigma(Q)\\
&\,\,\,\,\,-\int_{\partial D\cap B_{\delta}(P)}\mathrm{S.P.}[D_{m}^{(j)}](X,Q)n_{Q}^{(j)}f(Q)d\sigma(Q)\nonumber\\
&\triangleq \mathrm{I}_{1}^{(j)}-\mathrm{I}_{2}^{(j)}.\nonumber
\end{align}
Similarly to (3.21), by taking into account
$\mathrm{S.P.}[D_{m}^{(j)}](X,Q)\in
C(\Gamma_{\gamma,\eta}(P)\times\{Q\in \partial D:\,|Q-P|\leq
\delta\})$,
\begin{align}
\mathrm{I}_{2}^{(j)}\rightarrow \int_{\partial D\cap
B_{\delta}(P)}\mathrm{S.P.}[D_{m}^{(j)}](P,Q)n_{Q}^{(j)}f(Q)d\sigma(Q)
\,\,\mathrm{as} \, \,X\rightarrow P, \,\,X\in
\Gamma_{\gamma,\eta}(P).
\end{align}
For $X\in\Gamma_{\gamma,\eta}(P)$ and $|Q-P|<\delta<\frac{1}{2}$,
\begin{align}
D_{m}^{(j)}(X,Q)&=d_{m}\frac{|x_{j}-v_{j}|}{|X-Q|^{n+3-2m}}\\
&=d_{m}\frac{|x_{j}-v_{j}|}{\Big[|Q-P|^{2}+|X-P|^{2}-2(X-P)\cdot(Q-P)\Big]^{\frac{n+3}{2}-m}}\nonumber\\
&\leq
d_{m}\frac{|x_{j}-v_{j}|}{\Big[|Q-P|^{2}+|X-P|(1-2|Q-P|)\Big]^{\frac{n+3}{2}-m}}\nonumber\\
&\leq d_{m}\frac{|x_{j}-v_{j}|}{|Q-P|^{(n+3)-2m}},\nonumber
\end{align}
where $d_{m}=\frac{c_{n}}{\beta_{1}\beta_{2}\cdots\beta_{m-1}}.$
Therefore,
\begin{align}
|\mathrm{I}_{1}^{(j)}|&\leq d_{m}\int_{\partial D\cap
B_{\delta}(P)}|x_{j}-v_{j}|\frac{1}{|Q-P|^{(n+3)-2m}}|f(Q)|d\sigma(Q).
\end{align}
Since $2\leq (n+3)-2m\leq n-1$ (as $n=2$, we only need the second
inequality), then
\begin{align}
\mathrm{I}_{1}^{(j)}\rightarrow \int_{\partial D\cap
B_{\delta}(P)}D_{m}^{(j)}(P,Q)n_{Q}^{(j)}f(Q)d\sigma(Q)
\,\,\mathrm{as} \, \,X\rightarrow P, \,\,X\in
\Gamma_{\gamma,\eta}(P).
\end{align}
Therefore, in this case, by (3.20), (3.21), (3.23)-(3.25), (3.28),
\begin{equation*}\lim_{\substack{X\rightarrow P\\ X\in \Gamma_{\gamma}(P)}}\int_{\partial D}\langle K_{m}(X, Q), n_{Q}\rangle f(Q)d\sigma(Q)
        =\mathrm{K}_{m}f(P),\end{equation*}
        for any $f\in L^{p}(\partial D)$, $1\leq p\leq\infty$. \\
\textbf{Case 2}: $m\geq\frac{n+3}{2}$. For sufficiently large $T>0$,
we can split
\begin{align}
\int_{\partial D}K_{m}^{(j)}(X,Q)n_{Q}^{(j)}f(Q)d\sigma(Q)&=\int_{\partial D\cap B_{T}(P)}K_{m}^{(j)}(X,Q)n_{Q}^{(j)}f(Q)d\sigma(Q)\\
&\,\,\,\,\,\,+\int_{\partial D\setminus B_{T}(P)}K_{m}^{(j)}(X,Q)n_{Q}^{(j)}f(Q)d\sigma(Q)\nonumber\\
&\triangleq\mathrm{J}_{1}^{(j)}+\mathrm{J}_{2}^{(j)},\nonumber
\end{align}
where
\begin{align}
\mathrm{J}_{1}^{(j)}&=\int_{\partial D\cap B_{T}(P)}K_{m}^{(j)}(X,Q)n_{Q}^{(j)}f(Q)d\sigma(Q)\\
&=\int_{\partial D\cap
B_{T}(P)}D_{m}^{(j)}(X,Q)n_{Q}^{(j)}f(Q)d\sigma(Q)\\
&\,\,\,\,\,\,-\int_{\partial D\cap
B_{T}(P)}\mathrm{S. P.} [D_{m}^{(j)}](X,Q)n_{Q}^{(j)}f(Q)d\sigma(Q)\nonumber\\
&\triangleq\mathrm{J}_{11}^{(j)}-\mathrm{J}_{12}^{(j)}.\nonumber
\end{align}
Similarly to (3.23) and (3.25), we have
\begin{align}
\mathrm{J}_{2}^{(j)}\rightarrow \int_{\partial D\setminus
B_{T}(P)}K_{m}^{(j)}(P,Q)n_{Q}^{(j)}f(Q)d\sigma(Q) \,\,\mathrm{as}
\, \,X\rightarrow P, \,\,X\in \Gamma_{\gamma,\eta}(P)
\end{align}
and
\begin{align}
\mathrm{J}_{12}^{(j)}\rightarrow \int_{\partial D\cap
B_{T}(P)}\mathrm{S. P.} [D_{m}^{(j)}](P,Q)n_{Q}^{(j)}f(Q)d\sigma(Q)
\,\,\mathrm{as} \, \,X\rightarrow P, \,\,X\in
\Gamma_{\gamma,\eta}(P).
\end{align}
Since $m\geq\frac{n+3}{2}$, by (2.10) and (2.11),
$D_{m}^{(j)}(X,Q)\in
C(\Gamma_{\gamma,\eta}(P)\times\{Q\in \partial D:\,|Q-P|\leq T\}).$
Similarly to (3.28) (indeed, even more directly),
\begin{align}
\mathrm{J}_{11}^{(j)}\rightarrow \int_{\partial D\cap
B_{T}(P)}D_{m}^{(j)}(P,Q)n_{Q}^{(j)}f(Q)d\sigma(Q) \,\,\mathrm{as}
\, \,X\rightarrow P, \,\,X\in \Gamma_{\gamma,\eta}(P).
\end{align}
By (3.32)-(3.34), we have
\begin{equation*}\lim_{\substack{X\rightarrow P\\ X\in \Gamma_{\gamma}(P)}}\int_{\partial D}\langle K_{m}(X, Q), n_{Q}\rangle f(Q)d\sigma(Q)
        =\mathrm{K}_{m}f(P),\end{equation*}
        for any $f\in L^{p}(\partial D)$, $1\leq p\leq\infty$.

We thus conclude the claim (5) and the proof is complete.
\end{proof}

\subsection{$L^{p}$ boundedness properties of operators
$\mathrm{K}_{m}$ and multi-layer $\mathcal{D}$-potentials $M_{j}$}

In this section, we study the $L^{p}$ boundedness properties of the operators $\mathrm{K}_{m}$ given in (3.8) and the multi-layer $\mathcal{D}$-potentials $M_{j}$ defined by (3.2), which are very significant for the solving program in this paper.

To state the main results, we first introduce some necessary notions and notations which used thoroughly in the present section and what follows.

Let $D$ be a Lipschitz graph domain as in Theorem 3.5 and $w$ be a weight on $\partial D$, that is, a nonnegative locally integrable function on $\partial D$ with values in $(0,\infty)$ almost everywhere. For any $k,\alpha\geq 0$ and , if the weight $w$ on $\partial D$ satisfy
\begin{enumerate}
  \item $\left[|Q|^{k+\alpha}\left(1+\big|\log|Q|\big|\right)\right]w^{-1}(Q)\in L^{\infty}(\partial D)$;
  \item $\left[|Q|^{k}\left(1+\big|\log|Q|\big|\right)\right]^{p}|Q|^{\alpha-(p-1)n}w^{-1}(Q)\in L^{\frac{1}{p-1}}(\partial D)$ as $p\geq 1$,
\end{enumerate}
then $w$ is called to be a $(p,k,\alpha)$-weight on $\partial D$ and denote that $w\in \mathcal{W}^{p,k,\alpha}(\partial D)$ (Note that the above two conditions are the same as $p=1$). Here $\mathcal{W}^{p,k,\alpha}(\partial D)$ is the space consisting of all $(p,k,\alpha)$-weights on $\partial D$. It is easy to know that the spaces $\mathcal{W}^{p,k,\alpha}(\partial D)$ increases as $p$, $k$ and $\alpha$ decrease. That is,
\begin{prop}
Let $D$ be a Lipschitz graph domain as in Theorem 3.5, then
\begin{equation}\mathcal{W}^{p,k,\alpha}(\partial D)\subset\mathcal{W}^{p,l,\alpha}(\partial D)\subset\mathcal{W}^{q,l,\alpha}(\partial D)\subset\mathcal{W}^{q,l,\beta}(\partial D)\end{equation}
when $p>q>1$, $k>l$ and $\alpha>\beta$. Moreover, $\mathcal{W}^{1,k,\alpha}(\partial D)\subset\mathcal{W}^{1,l,\alpha}(\partial D)\subset\mathcal{W}^{1,l,\beta}(\partial D)$ as $k>l$ and $\alpha>\beta$.
\end{prop}

\begin{proof}
Note that $0\not\in\partial D$ and $|Q|\geq d_{0}$ for any $Q\in \partial D$. Therefore, when $k>l$, we have
\begin{align*}
&\int_{\partial D}\left\{\left[|Q|^{l}\left(1+\big|\log|Q|\big|\right)\right]^{p}|Q|^{\alpha}w^{-1}(Q)\right\}^{\frac{1}{p-1}}|Q|^{-n}d\sigma(Q)\\
=&\int_{\partial D}\frac{1}{|Q|^{\frac{k-l}{p-1}}}\left\{\left[|Q|^{k}\left(1+\big|\log|Q|\big|\right)\right]^{p}|Q|^{\alpha}w^{-1}(Q)\right\}^{\frac{1}{p-1}}|Q|^{-n}d\sigma(Q)\\
\leq & d_{0}^{-\frac{k-l}{p-1}}\int_{\partial D}\left\{\left[|Q|^{k}\left(1+\big|\log|Q|\big|\right)\right]^{p}|Q|^{\alpha}w^{-1}(Q)\right\}^{\frac{1}{p-1}}|Q|^{-n}d\sigma(Q)
\end{align*}
in which $p>1$; and similarly as $p=1$,
\begin{align*}
&\left[|Q|^{l+\alpha}\left(1+\big|\log|Q|\big|\right)\right]w^{-1}(Q)\leq d_{0}^{-(k-l)}\left[|Q|^{k+\alpha}\left(1+\big|\log|Q|\big|\right)\right]w^{-1}(Q).
\end{align*}
When $p>q>1$, we have
\begin{align*}
&\int_{\partial D}\left\{\left[|Q|^{l}\left(1+\big|\log|Q|\big|\right)\right]^{q}|Q|^{\alpha}w^{-1}(Q)\right\}^{\frac{1}{q-1}}|Q|^{-n}d\sigma(Q)\\
=&\int_{\partial D}\left[|Q|^{l}\left(1+\big|\log|Q|\big|\right)\right]^{\left(1+\frac{1}{q-1}\right)}\left(|Q|^{\alpha}w^{-1}(Q)\right)^{\frac{1}{q-1}}|Q|^{-n}d\sigma(Q)\\
=&\int_{\partial D}\left[|Q|^{l}\right]^{\left(1+\frac{1}{p-1}\right)}\left(|Q|^{\alpha}w^{-1}(Q)\right)^{\frac{1}{p-1}}\\
&\times \left\{\left[|Q|^{l+\alpha}\left(1+\big|\log|Q|\big|\right)\right]w^{-1}(Q)\right\}^{\left(\frac{1}{q-1}-\frac{1}{p-1}\right)}|Q|^{-n}d\sigma(Q)\\
\leq & \|w\|_{1, l, \alpha}^{\frac{p-q}{(p-1)(q-1)}}\int_{\partial D}\left\{\left[|Q|^{l}\left(1+\big|\log|Q|\big|\right)\right]^{p}|Q|^{\alpha}w^{-1}(Q)\right\}^{\frac{1}{p-1}}|Q|^{-n}d\sigma(Q),
\end{align*}
where
\begin{equation*}
\|w\|_{1, l, \alpha}=\sup_{Q\in\partial D}\left\{\left[|Q|^{l+\alpha}\left(1+\big|\log|Q|\big|\right)\right]w^{-1}(Q)\right\}.
\end{equation*}
When $\alpha>\beta$ and $q>1$, we have
\begin{align*}
&\int_{\partial D}\left\{\left[|Q|^{l}\left(1+\big|\log|Q|\big|\right)\right]^{q}|Q|^{\beta}w^{-1}(Q)\right\}^{\frac{1}{q-1}}|Q|^{-n}d\sigma(Q)\\
=&\int_{\partial D}|Q|^{\frac{\beta-\alpha}{q-1}}\left\{\left[|Q|^{l}\left(1+\big|\log|Q|\big|\right)\right]^{q}|Q|^{\alpha}w^{-1}(Q)\right\}^{\frac{1}{q-1}}|Q|^{-n}d\sigma(Q)\\
\leq & d_{0}^{-\frac{\alpha-\beta}{q-1}}\int_{\partial D}\left\{\left[|Q|^{l}\left(1+\big|\log|Q|\big|\right)\right]^{p}|Q|^{\alpha}w^{-1}(Q)\right\}^{\frac{1}{q-1}}|Q|^{-n}d\sigma(Q);
\end{align*}
and similarly as $q=1$,
\begin{align*}
&\left[|Q|^{l+\beta}\left(1+\big|\log|Q|\big|\right)\right]w^{-1}(Q)\leq d_{0}^{-(\alpha-\beta)}\left[|Q|^{l+\alpha}\left(1+\big|\log|Q|\big|\right)\right]w^{-1}(Q).
\end{align*}
Thus this proposition is completed.
\end{proof}

\begin{rem}
Moreover, by the condition (1) in the definition of $(p,k,\alpha)$-weights, it is easy to find that the weighted function spaces, $L^{p}(\partial D, wd\sigma)$ with $w\in \mathcal{W}^{p,k,\alpha}(\partial D)$, are subspaces of $L^{p}(\partial D)$ for any $1\leq p<\infty$ and $k,\alpha\geq 0$.
\end{rem}

Before stating the main results, we establish the following elementary and useful lemma.
\begin{lem}
Assume that $p\geq 1$ and $R\geq c_{0}$ with positive constant $c_{0}$ fixed, then
\begin{equation}
\int_{0}^{R}r|\log r|^{p}dr\leq CR^{2}\left[1+|\log R|^{p}\right]
\end{equation}
and
\begin{align}
\int_{R}^{\infty}\frac{|\log r|^{p}}{r^{2}}dr\leq C^{\prime}\frac{1}{\sqrt{R}}\left(1+|\log R|^{p}\right),
\end{align}
where the constants $C$ and $C^{\prime}$ depend only on $p$ and $c_{0}$.
\end{lem}

\begin{proof}
At first, we estimate (3.36). If $0<R<1$, then
\begin{align}
\int_{0}^{R}r|\log r|^{p}dr&=\int_{0}^{R}|r^{\frac{1}{p}}\log r|^{p}dr=\int_{0}^{R}\left(-r^{\frac{1}{p}}\log r\right)^{p}dr\\
&\leq \left(\frac{p}{e}\right)^{p}R\leq c_{0}^{-1}\left(\frac{p}{e}\right)^{p}R^{2}\nonumber\\
&\leq c_{0}^{-1}\left(\frac{p}{e}\right)^{p}R^{2}\left[1+|\log R|^{p}\right]\nonumber;
\end{align}
while $R\geq 1$, then
\begin{align}
\int_{0}^{R}r|\log r|^{p}dr&=\int_{0}^{1}\left(-r^{\frac{1}{p}}\log r\right)^{p}dr+\int_{1}^{R}r(\log r)^{p}dr\\
&\leq \int_{0}^{R}\left[\left(\frac{p}{e}\right)^{p}+R(\log R)^{p}\right]dr;\nonumber\\
&= \left(\frac{p}{e}\right)^{p}R+R^{2}(\log R)^{p}\nonumber\\
&\leq C_{p}R^{2}\left[1+|\log R|^{p}\right]\nonumber
\end{align}
where $C_{p}=\max\left\{\left(\frac{p}{e}\right)^{p},1\right\}$. So (3.36) follows.

Next turn to (3.37). If $R\geq 1$, we have that
\begin{align}
\int_{R}^{\infty}\frac{|\log r|^{p}}{r^{2}}dr&=\int_{R}^{\infty}\frac{(\log r)^{p}}{r^{2}}dr=\int_{R}^{\infty}\frac{(r^{-\frac{1}{2p}}\log r)^{p}}{r^{\frac{3}{2}}}dr\\
&\leq 2^{p+1}\left(\frac{p}{e}\right)^{p}\frac{1}{\sqrt{R}}\nonumber.
\end{align}
If $0<R<1$, then we have
\begin{align}
\int_{R}^{\infty}\frac{|\log r|^{p}}{r^{2}}dr&=\int_{1}^{\infty}\frac{(\log r)^{p}}{r^{2}}dr+\int_{R}^{1}\frac{(-\log r)^{p}}{r^{2}}dr\\
&\leq 2^{p+1}\left(\frac{p}{e}\right)^{p}+|\log R|^{p}\left(\frac{1}{R}-1\right)\nonumber\\
&\leq 2^{p+1}\left(\frac{p}{e}\right)^{p}+\frac{2}{R}|\log R|^{p}\nonumber\\
&\leq C_{p}^{\prime}\frac{1}{\sqrt{R}}\left(1+|\log R|^{p}\right)\nonumber
\end{align}
where $C_{p}^{\prime}=c_{0}^{-\frac{1}{2}}\max\left\{2^{p+1}\left(\frac{p}{e}\right)^{p},2\right\}$. Therefore, (3.37) follows from the last inequalities.
\end{proof}

\begin{rem}
By observing the above proof, in fact, we get that for any $0<\epsilon<1$,
\begin{align}
\int_{R}^{\infty}\frac{|\log r|^{p}}{r^{2}}dr\leq \widehat{C}\frac{1}{R^{1-\epsilon}}\left(1+|\log R|^{p}\right),
\end{align}
where the constant $\widetilde{C}$ depends only on $p,c_{0}$ and $\epsilon$ and satisfies that $\lim_{\epsilon\rightarrow 0+}\widehat{C}=+\infty$ and $\lim_{\epsilon\rightarrow 1-}\widehat{C}=+\infty$.

\end{rem}

The main object of this section is to justify

\begin{equation}
\mathrm{K}_{m}: \,\,\,L^{p}(\partial D, wd\sigma)\longrightarrow
L^{p}(\partial D)
\end{equation}
and
\begin{equation}
M_{j}: \,\,\,L^{p}(\partial D, w^{\prime}d\sigma)\longrightarrow
L^{p}(D)
\end{equation}
are bounded with $$\|\mathrm{K}_{m}f\|_{L^{p}(\partial D)}\leq
C\| f \|_{L^{p}(\partial D, wd\sigma)}$$ and $$\|M_{j}f\|_{L^{p}(D)}\leq
\widetilde{C}\| f \|_{L^{p}(\partial D, wd\sigma)}, $$ where $M_{j}$ is the $j$th-layer $\mathcal{D}$-potential, $w, w^{\prime}$ are appropriate $(p,k,\alpha)$-weights, and $C, \widetilde{C}$ are some constants depending only on $m,n,p$ and $D$. More precisely, we have

\begin{thm}
Let the Lipschitz graph domain $D$ and the operators $\mathrm{K}_{m}$, $m\geq 2$, be the same as in Theorem 3.5, $w\in \mathcal{W}^{p,2m-2,\frac{1}{2}}(\partial D)$, $1\leq p<\infty$, then
\begin{equation}\|\mathrm{K}_{m}f\|_{L^{p}(\partial D)}\leq
C\| f \|_{L^{p}(\partial D, wd\sigma)}\end{equation}
for any $f\in L^{p}(\partial D, wd\sigma)$, where $C$ is a constant depending only on $m, n, p$ and $d_{0}={\rm dist}(0, \partial D)$. That is, $\mathrm{K}_{m}$, $m\geq 2$, are bounded from $L^{p}(\partial D, wd\sigma)$ to $L^{p}(\partial D)$ for any $w\in \mathcal{W}^{p,2m-2,\frac{1}{2}}(\partial D)$ with $1\leq p<\infty$.
\end{thm}

\begin{proof}
By the definition of Lipschitz domain, we can identify the space $L^{p}(\partial D)$ with the weighted space $L^{p}\left(\mathbb{R}^{n}, \sqrt{1+|\nabla\varphi|^{2}}dx\right)$. It is easy to verify that the space is comparable the standard space $L^{p}(\mathbb{R}^{n})$ due to the following inequalities
\begin{equation}\|f\|_{L^{p}(\mathbb{R}^{n})}\leq\|f\|_{L^{p}\left(\mathbb{R}^{n},\sqrt{1+|\nabla\varphi|^{2}}dx\right)}\leq \sqrt{1+L^{2}}\|f\|_{L^{p}(\mathbb{R}^{n})},\end{equation}
where $L$ is the Lipschitz constant of $D$. So here we can simply regard $L^{p}(\mathbb{R}^{n})$ as $L^{p}\left(\mathbb{R}^{n}, \sqrt{1+|\nabla\varphi|^{2}}dx\right)$  identically. Similarly, we can also identify $L^{p}(\partial D, wd\sigma)$ with $L^{p}(\mathbb{R}^{n}, wdx)$.

For simplicity, we will use the spaces $L^{p}(\mathbb{R}^{n})$ and $L^{p}(\mathbb{R}^{n}, wdx)$ to replace the spaces $L^{p}(\partial D)$ and $L^{p}(\partial D, wd\sigma)$ in the following argument.

By Minkowski's inequality (also for integrals) and H\"older's inequality, we have
\begin{align}\|\mathrm{K}_{m}f\|_{L^{p}(\mathbb{R}^{n})}&\leq C\sum_{j=1}^{n+1}\left\{\int_{\mathbb{R}^{n}}\left[\int_{\mathbb{R}^{n}}|K_{m}^{(j)}(x,y)f(y)|dy\right]^{p}dx\right\}^{\frac{1}{p}}\\
&\leq C\sum_{j=1}^{n+1}\int_{\mathbb{R}^{n}}\left[\int_{\mathbb{R}^{n}}|K_{m}^{(j)}(x,y)|^{p}dx\right]^{\frac{1}{p}}|f(y)|dy\nonumber\\
&\leq C\sum_{j=1}^{n+1}\int_{\mathbb{R}^{n}}\left[|y|^{2m-2}\left(1+\big|\log|y|\big|\right)\right]|y|^{\frac{1}{p}\left[\frac{1}{2}-(p-1)n\right]}|f(y)|dy\nonumber\\
&\leq C\sum_{j=1}^{n+1}\left\{\int_{\mathbb{R}^{n}}\left[|y|^{2m-2}\left(1+\big|\log|y|\big|\right)\right]^{\frac{p}{p-1}}\left(|y|^{\frac{1}{2}}w^{-1}(y)\right)^{\frac{1}{p-1}}|y|^{-n}dy\right\}^{\frac{p-1}{p}}\nonumber\\
&\,\,\,\,\,\,\times\|f\|_{L^{p}(\mathbb{R}^{n}, wdx)}\nonumber\\
&= C\sum_{j=1}^{n+1}\left\{\int_{\mathbb{R}^{n}}\left[\left[|y|^{2m-2}\left(1+\big|\log|y|\big|\right)\right]^{p}|y|^{\frac{1}{2}}w^{-1}(y)\right]^{\frac{1}{p-1}}|y|^{-n}dy\right\}^{\frac{p-1}{p}}\nonumber\\
&\,\,\,\,\,\,\times\|f\|_{L^{p}(\mathbb{R}^{n}, wdx)}\nonumber\\
&\leq C\|f\|_{L^{p}(\mathbb{R}^{n}, wdx)},\nonumber
\end{align}
since $w\in \mathcal{W}^{p,2m-2,\frac{1}{2}}(\partial D)$, where the constant $C$ depends only on $m,n,p,d_{0}$, and the following fact
\begin{equation}
\int_{\mathbb{R}^{n}}|K_{m}^{(j)}(x,y)|^{p}dx\leq C(m,n,d_{0}) \left[|y|^{2m-2}\left(1+\big|\log|y|\big|\right)\right]^{p}|y|^{\frac{1}{2}-(p-1)n}
\end{equation}
has been used in the third inequality, in which the constant $C(m,n,d_{0})$ depends only on $m,n$ and $d_{0}$.

To get (3.48), we split
\begin{align}
\int_{\mathbb{R}^{n}}|K_{m}^{(j)}(x,y)|^{p}dx&=\int_{|x|\leq 2|y|}|K_{m}^{(j)}(x,y)|^{p}dx+\int_{|x|> 2|y|}|K_{m}^{(j)}(x,y)|^{p}dx\\
&=\mathcal{I}_{1}(y)+\mathcal{I}_{2}(y).\nonumber
\end{align}
Note that
\begin{align}
\mathcal{I}_{1}(y)&\leq C_{p}\left\{\int_{|x|\leq 2|y|}|D_{m}^{(j)}(x,y)|^{p}dx+\left(\int_{|x|<|y|}+\int_{|y|<|x|<2|y|}\right)|\mathrm{S. P.} [D_{m}^{(j)}](x,y)|^{p}dx\right\}\\
&=\mathcal{I}_{1,1}(y)+\mathcal{I}_{1,2}(y)+\mathcal{I}_{1,3}(y).\nonumber
\end{align}

Firstly, to estimate $\mathcal{I}_{1,1}$, we note that $|x-y|\leq 3|y|$ when $|x|\leq 2|y|$, and
\begin{equation}
|D_{m}^{(j)}(x,y)|\leq C_{m,n}|x-y|^{2m-(n+2)}\left(1+\big|\log|x-y|\big|\right),
\end{equation}
where $C_{m,n}$ is a constant depending only $m$ and $n$, then by (3.36),
\begin{align}
\int_{|x|\leq 2|y|}|D_{m}^{(j)}(x,y)|^{p}dx&\leq \int_{|x-y|\leq 3|y|}|D_{m}^{(j)}(x,y)|^{p}dx\\
&\leq C(m, n)\int_{|x-y|\leq 3|y|}\left[|x-y|^{2m-(n+2)}\left(1+\big|\log|x-y|\big|\right)\right]^{p}dx\nonumber\\
&\leq C(m, n, p)|y|^{(2m-(n+2))p}\int_{|x-y|\leq 3|y|}\left[\left(1+\big|\log|x-y|\big|\right)\right]^{p}dx\nonumber\\
&\leq C(m, n, p)|y|^{(2m-2)p-(p-1)n}\left(1+\big|\log|y|\big|\right)^{p}\nonumber\\
&\leq C(m, n, p, d_{0})\left[|y|^{(2m-2)}\left(1+\big|\log|y|\big|\right)\right]^{p}|y|^{\frac{1}{2}-(p-1)n},\nonumber
\end{align}
where $C(\cdots)$ denotes a constant depending only on the parameters in the parenthesis, and the fact $|y|\geq d_{0}$ have been used in the last inequality.

Next to estimate $\mathcal{I}_{1,2}$. When $|x|<|y|$, by the definition
\begin{align}
|\mathrm{S. P.} [D_{m}^{(j)}](x,y)|\leq C_{m,n}|y|^{2m-n-2}\left(1+\big|\log|y|\big|\right),
\end{align}
we have
\begin{align}
\int_{|x|< |y|}|\mathrm{S. P.} [D_{m}^{(j)}](x,y)|^{p}dx&\leq C(m,n)\left[|y|^{2m-n-2}\left(1+\big|\log|y|\big|\right)\right]^{p}{\rm Vol}(B(0, |y|))\\
&\leq C(m, n)\left[|y|^{2m-2}\left(1+\big|\log|y|\big|\right)\right]^{p}|y|^{-(p-1)n}\nonumber\\
&\leq C(m, n, d_{0})\left[|y|^{2m-2}\left(1+\big|\log|y|\big|\right)\right]^{p}|y|^{\frac{1}{2}-(p-1)n},\nonumber
\end{align}
where the fact $|y|\geq d_{0}$ have been used in the last inequality.

The third to estimate $\mathcal{I}_{1,3}$. In this case, by the definition, as $|y|<|x|<2|y|$,
\begin{align}
|\mathrm{S. P.} [D_{m}^{(j)}](x,y)|\leq C_{m,n}|y|^{2m-n-2}\left(1+\big|\log|x|\big|\right),
\end{align}
then by (3.36), we obtain
\begin{align}
\int_{|y|<|x|<2|y|}|\mathrm{S. P.} [D_{m}^{(j)}](x,y)|^{p}dx&\leq C(m,n)|y|^{(2m-n-2)p}\int_{|y|<|x|<2|y|}(1+\big|\log|x|\big|)^{p}dx\\
&\leq C(m, n, p)\left[|y|^{2m-2}\left(1+\big|\log|y|\big|\right)\right]^{p}|y|^{-(p-1)n}\nonumber\\
&\leq C(m, n, p, d_{0})\left[|y|^{2m-2}\left(1+\big|\log|y|\big|\right)\right]^{p}|y|^{\frac{1}{2}-(p-1)n},\nonumber
\end{align}
where the fact $|y|\geq d_{0}$ have been used in the last inequality.

Finally, we turn to estimate $\mathcal{I}_{2}$.  Note that $r=\frac{|y|}{|x|}\in (0,\frac{1}{2})$ as $|x|>2|y|$, and $1-2r(x_{S^{n}}\cdot y_{S^{n}})+r^{2}\in (\frac{1}{4}, \frac{9}{4})$ as $r\in (0, \frac{1}{2})$. Thus by (3.11)-(3.14) and the definition, we have
\begin{align}
|\mathrm{I. P.} [D_{m}^{(j)}](x,y)|\leq C_{m,n}|y|^{2m}\left(1+\big|\log|x|\big|\right)\frac{1}{|x|^{n+2}}.
\end{align}
Therefore, by (3.37), we get
\begin{align}
\int_{|x|> 2|y|}|K_{m}^{(j)}(x,y)|dx&= \int_{|x|> 2|y|}|\mathrm{I. P.} [D_{m}^{(j)}](x,y)|dx\\
&\leq C(m, n)|y|^{2mp}\int_{|x|> 2|y|}\left[\frac{1+\big|\log|x|\big|}{|x|^{n+2}}\right]^{p}dx\nonumber\\
&\leq C(m, n, d_{0})|y|^{2mp-(n+2)(p-1)}\int_{|x|> 2|y|}\frac{\left(1+\big|\log|x|\big|\right)^{p}}{|x|^{n+2}}dx\nonumber\\
&=\leq C(m, n, d_{0})|y|^{(2m-1)p-(n+1)(p-1)}\int_{|x|> 2|y|}\frac{\left(1+\big|\log|x|\big|\right)^{p}}{|x|^{n+1}}dx\nonumber\\
&\leq C(m, n, p, d_{0})|y|^{(2m-2)p-(p-1)n+1}\left(\frac{1}{|y|}+\frac{\big|\log|y|\big|^{p}}{\sqrt{|y|}}\right)\nonumber\\
&\leq C(m, n, p, d_{0})\left[|y|^{2m-2}\left(1+|\log|y||\right)\right]^{p}|y|^{-(p-1)n+\frac{1}{2}}\nonumber
\end{align}
where the fact $|y|\geq d_{0}$ have been used in the last inequality.

Therefore, (3.48) follows from (3.49), (3.50), (3.52), (3.54), (3.56), (3.58). Thus the theorem is completed.
\end{proof}

\begin{thm}
Let the graph Lipschitz domain $D$ and operators $M_{j}$, $j\geq 2$, be the same as in Theorem 3.5, $w\in \mathcal{W}^{p,2j-2,\frac{3}{2}}(\partial D)$, $1\leq p<\infty$, then
\begin{equation}\|M_{j}f\|_{L^{p}(D)}\leq
C\| f \|_{L^{p}(\partial D, wd\sigma)}\end{equation}
for any $f\in L^{p}(\partial D, wd\sigma)$, where $C$ is a constant depending only on $m, n, p$ and $d_{0}$. That is, $M_{j}$, $j\geq 2$, are bounded from $L^{p}(\partial D, wd\sigma)$ to $L^{p}(D)$ for any $w\in \mathcal{W}^{p,2j-2,\frac{3}{2}}(\partial D)$ with $1\leq p<\infty$.
\end{thm}

\begin{proof}
It is similar to Theorem 3.11 only with $X\in D$ in place of $P\in\partial D$.
\end{proof}

\section{Polyharmonic Dirichlet problems in Lipschitz graph domains}

In this section, we solve the PHD problems (1.1), viz.,

\begin{equation}\begin{cases}\Delta^{m}u=0,\,\,in\,\, D,\\
  \Delta^{j}u=f_{j},\,\,on\,\,\partial D,
\end{cases}
  \end{equation}
where $u-M_{1}\widetilde{f}_{0}\in L^{p}(D)$ with $\|u-M_{1}\widetilde{f}_{0}\|_{L^{p}(D)}\leq C\sum_{j=1}^{m-1}\|f_{j}\|_{L^{p}(\partial D, wd\sigma)}$ in which the constant $C$ depends only on $m, n, p$ and $d_{0}$,
$\Delta=\sum_{k=1}^{n+1}\frac{\partial^{2}}{\partial
x_{k}^{2}}$, $D$ is a Lipschitz graph domain stated as in Theorem 3.5, $f_{0}\in
L^{p}(\partial D)$ and $f_{j}\in
L^{p}(\partial D, wd\sigma)$, $1\leq j\leq m-1$ for some suitable $p>1$, the $(p,2m-2,\frac{3}{2})$-weight $w$ on $\partial D$ is given as in section 3.1, $\widetilde{f}_{0}$ is related to all the boundary data $f_{j}$, $0\leq j<m$ and $m\in\mathbb{N}$.

To do so, firstly, we establish

\begin{lem}
Let $E$ be a simply connected unbounded domain in $\mathbb{R}^{n+1}$
with smooth boundless boundary $\partial E$. If $f\in (C^{1}\times
C)\left((\mathbb{R}^{n+1}\setminus \partial E)\times\partial
E\right)$ and there exist $g_{0}, g_{1}\in L^{p}(\partial E)$, $p\geq 1$
such that
\begin{equation}
|f(X,Q)|\leq M_{0}\frac{g_{0}(Q)}{(1+|Q|^{2})^{\frac{n}{2}}}
\end{equation}
and
\begin{equation}
|\frac{\partial}{\partial x_{j}}f(X,Q)|\leq
M_{1}\frac{g_{1}(Q)}{(1+|Q|^{2})^{\frac{n}{2}}}
\end{equation}
hold for any $(X,Q)\in E_{c}\times\{Q\in\partial E: |Q|>T\}$ and
$j=1,2,\ldots,n+1$, where $E_{c}$ is a compact subset of
$\mathbb{R}^{n+1}\setminus \partial E$, $T$ is a sufficiently large
positive real number and $M_{0}, M_{1}$ are positive constants
depending only on $E_{c}$ and $T$, then
\begin{equation}
\frac{\partial}{\partial x_{j}}\left(\int_{\partial
E}f(X,Q)d\sigma(Q)\right)=\int_{\partial E}\frac{\partial
f}{\partial x_{j}}(X,Q)d\sigma(Q),\,\,X\in \mathbb{R}^{n+1}\setminus
\partial E
\end{equation}
for any $1\leq j\leq n+1$, where $d\sigma$ is the surface measure of $\partial E$.
\end{lem}

\begin{proof}
Fix $X=(x_{1},x_{2},\ldots,x_{n+1})\in E$ and
$j\in\{1,2,\ldots,n+1\}$, take $X_{l}=X+t_{l}e_{j}$ with
$\lim_{l\rightarrow+\infty}t_{l}=0$, and
$e_{j}=(0,\ldots,1,\ldots,0)\in \mathbb{R}^{n+1}$ whose the $j$th
element is $1$ and other ones are zero. Denote

\begin{eqnarray}
D_{l}(X,Q)&=&\frac{f(X_{l},Q)-f(X,Q)}{t_{l}}\\
              &=&\frac{\partial}{\partial x_{j}}f(X+\theta t_{l}e_{j},Q)\nonumber,
\end{eqnarray}
where $0<\theta<1$, then by (4.3),
\begin{eqnarray}
|D_{l}(X,Q)|\leq M_{1}\frac{g_{1}(Q)}{(1+|Q|^{2})^{\frac{n}{2}}}
\end{eqnarray}
uniformly in $\{Q\in\partial E: |Q|>T\}$ whenever $X_{l}\in\{Y:
|Y-X|\leq R\}\subset \mathbb{R}^{n+1}\setminus \partial E$ for some
$R>0$ and sufficiently large $T>0$. Since $f\in (C^{1}\times
C)\left((\mathbb{R}^{n+1}\setminus \partial E)\times\partial
E\right)$ and
\begin{equation}
\lim_{l\rightarrow+\infty}D_{l}(X,Q)=\frac{\partial f}{\partial
x_{j}}(X,Q),\,\,Q\in\partial E,
\end{equation}
by (4.2), (4.6), the continuity of $f$ on compact set $\{Y:
|Y-X|\leq R\}\times\{Q\in\partial D: |Q|\leq T\}$, and Lebesgue's
dominated convergence theorem,
\begin{align}
\lim_{l\rightarrow+\infty}\int_{\partial E}D_{l}(X,Q)d\sigma(Q)
=&\lim_{l\rightarrow+\infty}\Big[\int_{|Q|\leq T, Q\in\partial E}D_{l}(X,Q)d\sigma(Q)\\
&\,\,+\int_{|Q|> T, Q\in\partial E}D_{l}(X,Q)d\sigma(Q)\Big]\nonumber\\
=&\int_{|Q|\leq T, Q\in\partial E}\frac{\partial f}{\partial
x_{j}}(X,Q)d\sigma(Q)\nonumber\\
&\,\,+\int_{|Q|> T, Q\in\partial E}\frac{\partial
f}{\partial x_{j}}(X,Q)d\sigma(Q)\nonumber\\
=&\int_{\partial E}\frac{\partial f}{\partial
x_{j}}(X,Q)d\sigma(Q),\nonumber
\end{align}
i.e.,
\begin{eqnarray*}
\lim_{l\rightarrow+\infty}\frac{\int_{\partial
E}f(X_{l},Q)d\sigma(Q)-\int_{\partial E}f(X,Q)d\sigma(Q)}{t_{l}}=\int_{\partial
E}\frac{\partial f}{\partial x_{j}}(X,Q)d\sigma(Q).
\end{eqnarray*}
Since $X$ and the sequence $X_{l}$ are arbitrarily chosen, then
\begin{equation*}
\frac{\partial}{\partial x_{j}}\left(\int_{\partial
E}f(X,Q)d\sigma(Q)\right)=\int_{\partial E}\frac{\partial f}{\partial
x_{j}}(X,Q)d\sigma(Q)
\end{equation*}
for any $1\leq j\leq n+1$ and $X\in \mathbb{R}^{n+1}\setminus
\partial E$.
\end{proof}

As an immediate consequence, we have
\begin{cor}
Let $E$ be a simply connected unbounded domain in $\mathbb{R}^{n+1}$
with smooth boundless boundary $\partial E$. If $f\in (C^{2}\times
C)\left((\mathbb{R}^{n+1}\setminus \partial E)\times\partial
E\right)$ and there exist $g_{0}, g_{1},g_{2}\in L^{p}(\partial E)$,
$p\geq1$ such that
\begin{equation}
|f(X,Q)|\leq M_{0}\frac{g_{0}(Q)}{(1+|Q|^{2})^{\frac{n}{2}}},
\end{equation}
\begin{equation}
|\frac{\partial}{\partial x_{j}}f(X,Q)|\leq
M_{1}\frac{g_{1}(Q)}{(1+|Q|^{2})^{\frac{n}{2}}}
\end{equation}
and
\begin{equation}
|\frac{\partial^{2}}{\partial x_{j}^{2}}f(X,Q)|\leq
M_{2}\frac{g_{2}(Q)}{(1+|Q|^{2})^{\frac{n}{2}}}
\end{equation}
hold for any $(X,Q)\in E_{c}\times\{Q\in\partial E: |Q|>T\}$ and
$j=1,2,\ldots,n+1$, where $E_{c}$ is any compact subset of
$\mathbb{R}^{n+1}\setminus \partial E$, $T$ is a sufficiently large
positive real number and $M_{0}, M_{1}, M_{2}$ are positive
constants depending only on $E_{c}$ and $T$, then
\begin{equation}
\Delta\left(\int_{\partial E}f(X,Q)d\sigma(Q)\right)=\int_{\partial
E}\Delta f(X,Q)d\sigma(Q),\,\,X\in \mathbb{R}^{n+1}\setminus
\partial E.
\end{equation}
\end{cor}

From the above corollary, we can obtain the following theorem
concerning the differentiability of the multi-layer $\mathcal{D}$-potentials.

\begin{thm}
Let $\{\,K_{m}\,\}_{m=1}^{\infty}$ be the sequence of higher order
Poisson fields as in the previous section, and $E$ be a simply
connected unbounded domain in $\mathbb{R}^{n+1}$ with smooth
boundless boundary $\partial E$. Then for any $m>1$ and $f\in
L^{p}(\partial E)$, $p\geq1$,
\begin{equation}
\Delta\left(\int_{\partial E}\langle K_{m}(X, Q), n_{Q}\rangle
f(Q)d\sigma(Q)\right)=\int_{\partial E}\langle K_{m-1}(X, Q),
n_{Q}\rangle f(Q)d\sigma(Q),
\end{equation}
where $X\in \mathbb{R}^{n+1}\setminus \partial E$, namely,
\begin{equation}
\Delta M_{m}f(X)=M_{m-1}f(X),\,\, X\in \mathbb{R}^{n+1}\setminus \partial E.
\end{equation}
\end{thm}

\begin{proof}
From the claim (1) in Theorem 3.5 (by the same argument,
the claims (1)-(3) and (5) make sense for the present domains $E$
stated here), we know that  $K_{m}\in (C^{2}\times
C)\left((\mathbb{R}^{n+1}\setminus \partial E)\times
\partial E\right)$. For any $1\leq j\leq n+1$,
\begin{align}K_{m}^{(j)}(X,Q)=&D_{m}^{(j)}(X,Q)-\mathrm{S. P.} [D_{m}^{(j)}](X,Q)=\mathrm{I. P.}
[D_{m}^{(j)}](X,Q)\\
=&(x_{j}-v_{j})\sum_{k=2m}^{\infty}[C_{m,-k}(X,Q)+\widetilde{C}_{m,-k}(X,Q)\log|Q|]\frac{1}{(1+|Q|^{2})^{\frac{k}{2}-m+\frac{n+3}{2}}}\nonumber,\end{align}
for any $(X,Q)\in (\mathbb{R}^{n+1}\setminus \partial
E)\times\partial E$ with $|X|<|Q|$, where $C_{m,-k}$ and $\widetilde{C}_{m,-k}$ can be
explicitly expressed by the ultraspherical polynomials
$P_{l}^{(\frac{n+3}{2}-m)}$ and $Q_{l}^{(\frac{n+3}{2}-m)}$. So by the claim
(2) in Theorem 3.5, i.e., (3.16) and similar arguments to (3.16), we
obtain
\begin{equation}
|K_{m}^{(j)}(X,Q)|\leq M_{0}\frac{1}{(1+|Q|^{2})^{\frac{n+1+\epsilon}{2}}},
\end{equation}
\begin{equation}
|\frac{\partial}{\partial x_{l}}K_{m}^{(j)}(X,Q)|\leq
M_{1}\frac{1}{(1+|Q|^{2})^{\frac{n+1+\epsilon}{2}}}
\end{equation}
and
\begin{equation}
|\frac{\partial^{2}}{\partial x_{l}^{2}}K_{m}^{(j)}(X,Q)|\leq
M_{2}\frac{1}{(1+|Q|^{2})^{\frac{n+1+\epsilon}{2}}}
\end{equation}
for any $m\geq 2$, $1\leq l\leq n+1$, $0<\epsilon<1$, and $(X,Q)\in
E_{c}\times\{Q\in\partial E: |Q|>T\}$, where $E_{c}$ is any compact
subset of $\mathbb{R}^{n+1}\setminus \partial E$, $T$ is a
sufficiently large positive real number and $M_{0},M_{1},M_{2}$ are
positive constants depending only on $E_{c}, T$ and $\epsilon$. Therefore, by a similar argument as
Corollary 4.2 and the claim (3) in Theorem 3.5, for any $m>1$,
\begin{equation}
\Delta\left(\int_{\partial E}\langle K_{m}(X, Q), n_{Q}\rangle
f(Q)d\sigma(Q)\right)=\int_{\partial E}\langle K_{m-1}(X, Q),
n_{Q}\rangle f(Q)d\sigma(Q),
\end{equation}
where $X\in \mathbb{R}^{n+1}\setminus
\partial E$, i.e.,
\begin{equation*}
\Delta M_{m}f(X)=M_{m-1}f(X),\,\, X\in \mathbb{R}^{n+1}\setminus
\partial E. \qedhere
\end{equation*}
\end{proof}

\begin{rem}
By the same arguments, all of the above results hold when the
domains $E$ are replaced by the Lipschitz graph domains $D$ stated as in Theorem 3.5.
\end{rem}

Now we can give the main result for polyharmonic Dirichlet problems
in Lipschitz graph domains as follows.

\begin{thm}
Let $\{\,K_{m}\,\}_{m=1}^{\infty}$ be the sequence of the Poisson
fields, and $D$ be a Lipschitz graph domain in
$\mathbb{R}^{n+1}$ with Lipschitz graph boundary $\partial D$ as in Theorem 3.5, then
for any $m>1$, there exists $\varepsilon=\varepsilon(D)>0$ such that
the PHD problem (4.1) with the data $f_{0}\in L^{p}(\partial D)$ and $f_{j}\in L^{p}(\partial D, wd\sigma)$ with $w\in W^{p, 2m-2,\frac{3}{2}}(\partial D)$, $1\leq j\leq m-1$, $2-\varepsilon<p<\infty$, is solvable and a
solution is given by
\begin{align}
u(X)&=\sum_{j=1}^{m}\int_{\partial D}\langle K_{j}(X, Q),
n_{Q}\rangle \widetilde{f}_{j-1}(Q)d\sigma(Q),\\
&=\sum_{j=1}^{m}M_{j}\widetilde{f}_{j-1}(X),\,\,\,\, X\in D,\nonumber
\end{align}
where
\begin{equation}\widetilde{f}_{m-1}=\left(\frac{1}{2}I+T\right)^{-1}f_{m-1}\end{equation} and
\begin{equation}\widetilde{f}_{l}=\left(\frac{1}{2}I+T\right)^{-1}\left(f_{l}-\sum_{j=l+2}^{m}\mathrm{K}_{j-l}\widetilde{f}_{j-1}\right)\end{equation}
with $0\leq l\leq m-2$, which satisfying the following estimate
\begin{equation}\|u-M_{1}\widetilde{f}_{0}\|_{L^{p}(D)}\leq C\sum_{j=1}^{m-1}\|f_{j}\|_{L^{p}(\partial D, wd\sigma)}.\end{equation}
Under this estimate, the solution (4.20) with (4.21) and (4.22) is unique.

\end{thm}

\begin{proof}
At first, we consider the existence of solution to (4.1). Formally, denote the solution of (4.1) as follows
\begin{equation}
u(X)=M_{1}\widetilde{f}_{0}(X)+M_{2}\widetilde{f}_{1}(X)+\cdots+M_{m}\widetilde{f}_{m-1}(X)
\end{equation}
for some functions $\widetilde{f}_{j}$, $0\leq j\leq m-1$ to be determined soon, where $M_{j}$ is the $j$th-layer $\mathcal{D}$-potential.

Letting the polyharmonic operators $\Delta^{l}$, $0\leq
l\leq m$, acting on two sides of (4.24), by Theorem 4.3, we formally have
\begin{align}
\begin{cases}
u(X)&=M_{1}\widetilde{f}_{0}(X)+M_{2}\widetilde{f}_{1}(X)+M_{3}\widetilde{f}_{2}(X)+\cdots+M_{m}\widetilde{f}_{m-1}(X),\\
\Delta u(X)&=M_{1}\widetilde{f}_{1}(X)+M_{2}\widetilde{f}_{2}(X)+\cdots+M_{m-1}\widetilde{f}_{m-1}(X),\nonumber\\
\Delta^{2} u(X)&=M_{1}\widetilde{f}_{2}(X)+\cdots+M_{m-2}\widetilde{f}_{m-1}(X),\nonumber\\
&\cdots\nonumber\\
\Delta^{m-1} u(X)&=M_{1}\widetilde{f}_{m-1}(X),\nonumber\\
\Delta^{m} u(X)&=0.
\end{cases}
\end{align}
Furthermore, let $X\in D$ converge to $P\in\partial D$ non-tangentially , by (3.6) and (3.7), using the boundary value data of (4.1), then
\begin{align}
\begin{cases}
f_{0}(P)&=\left(\frac{1}{2}I+T\right)\widetilde{f}_{0}(P)+\mathrm{K}_{2}\widetilde{f}_{1}(P)+\mathrm{K}_{3}\widetilde{f}_{2}(P)+\cdots+\mathrm{K}_{m}\widetilde{f}_{m-1}(P),\\
f_{1}(P)&=\left(\frac{1}{2}I+T\right)\widetilde{f}_{1}(P)+\mathrm{K}_{2}\widetilde{f}_{2}(P)+\cdots+\mathrm{K}_{m-1}\widetilde{f}_{m-1}(P),\nonumber\\
f_{2}(P)&=\left(\frac{1}{2}I+T\right)\widetilde{f}_{2}(P)+\cdots+\mathrm{K}_{m-2}\widetilde{f}_{m-1}(P),\nonumber\\
&\cdots\nonumber\\
f_{m-1}(P)&=\left(\frac{1}{2}I+T\right)\widetilde{f}_{m-1}(P).\nonumber\\
\end{cases}
\end{align}

By the invertible property of $\frac{1}{2}I+T$ and $L^{p}$ boundness of $\mathrm{K}_{m}$, then we have
\begin{align}
\begin{cases}
\widetilde{f}_{0}(P)&=\left(\frac{1}{2}I+T\right)^{-1}\left[f_{0}(P)-\mathrm{K}_{2}\widetilde{f}_{1}(P)-\mathrm{K}_{3}\widetilde{f}_{2}(P)-\cdots-\mathrm{K}_{m}\widetilde{f}_{m-1}(P)\right],\\
\widetilde{f}_{1}(P)&=\left(\frac{1}{2}I+T\right)^{-1}\left[f_{1}(P)-\mathrm{K}_{2}\widetilde{f}_{2}(P)-\cdots-\mathrm{K}_{m-1}\widetilde{f}_{m-1}(P)\right],\nonumber\\
\widetilde{f}_{2}(P)&=\left(\frac{1}{2}I+T\right)^{-1}\left[f_{2}(P)-\cdots-\mathrm{K}_{m-2}\widetilde{f}_{m-1}(P)\right],\nonumber\\
&\cdots\nonumber\\
\widetilde{f}_{m-1}(P)&=\left(\frac{1}{2}I+T\right)^{-1}f_{m-1}(P).\nonumber\\
\end{cases}
\end{align}

Therefore, we get
\begin{align}
\begin{cases}
\widetilde{f}_{m-1}=\left(\frac{1}{2}I+T\right)^{-1}f_{m-1},\vspace{2mm}\\
\widetilde{f}_{l}=\left(\frac{1}{2}I+T\right)^{-1}\left[f_{l}-\sum_{j=l+2}^{m}\mathrm{K}_{j-l}\widetilde{f}_{j-1}\right],
\end{cases}
\end{align}
where $0\leq l\leq m-2$.
More concisely,
\begin{equation}
\widetilde{f}_{l}=\left(\frac{1}{2}I+T\right)^{-1}\left(f_{l}-\sum_{j=l+2}^{m}\mathrm{K}_{j-l}\widetilde{f}_{j-1}\right)
\end{equation}
with $0\leq l\leq m-1$ by the convention that $\sum_{j=l}^{k}s_{j}=0$ as $k<l$.

Noting Remark 3.8, by Lemma 3.4 and Theorem 3.11, it is noteworthy that the above formal reasoning makes sense when $f_{0}\in L^{p}(\partial D)$ and $f_{j}\in L^{p}(\partial D, wd\sigma)$ with $w\in \mathcal{W}^{p, 2m-2,\frac{3}{2}}(\partial D)$, $1\leq j\leq m-1$, $2-\varepsilon<p<\infty$, where $\varepsilon$ is the same as in Lemma 3.4. That is, a solution of (4.1) is (4.20) with (4.21) and (4.22).

Next we turn to the estimate and uniqueness of the solution. By Theorems 3.11, 3.12, and Lemma 3.4, we have
\begin{align}\|u-M_{1}\widetilde{f}_{0}\|_{L^{p}(D)}&=\|\sum_{j=2}^{m}M_{j}\widetilde{f}_{j-1}\|_{L^{p}(D)}\\
&\leq \sum_{j=2}^{m}\| M_{j}\widetilde{f}_{j-1} \|_{L^{p}(D)}\nonumber\\
&\leq C\sum_{j=1}^{m-1}\| f_{j} \|_{L^{p}(\partial D,\,wd\sigma)}\nonumber
\end{align}
where $w\in \mathcal{W}^{p, 2m-2,\frac{3}{2}}(\partial D)$ with $2-\varepsilon<p<\infty$, and the constant $C$ depends only on $m,n,p$ and $d_{0}$.

So by the above estimate, the uniqueness of solution follows. Thus this theorem is completed.
\end{proof}

\section{Polyharmonic fundamental solutions}
By similar computations as in Section 2,  it is easy to know that
\begin{equation*}
\Delta \left(|x|^{s}\right)=s(s+n-1)|x|^{s-2},
\end{equation*}
\begin{equation*}
\Delta
\left(|x|^{s}\log|x|\right)=s(s+n-1)|x|^{s-2}\log|x|+(2s+n-1)|x|^{s-2}
\end{equation*}
and
\begin{equation*}
\Delta \left(\log|x|\right)=(n-1)|x|^{-2}.
\end{equation*}

Set \begin{equation}
\delta_{s}=s(s+n-1),
\end{equation}
therefore
\begin{equation}
\Delta\left(\frac{1}{\delta_{s}}|x|^{s}\right)=|x|^{s-2},
\end{equation}

\begin{equation}
\Delta\left(\frac{1}{\delta_{s}}|x|^{s}\log|x|\right)=|x|^{s-2}\log|x|+\left(\frac{1}{s}+\frac{1}{s+n-1}\right)|x|^{s-2}
\end{equation}
and
\begin{equation}
\Delta\left(\frac{1}{n-1}\log|x|\right)=|x|^{-2}.
\end{equation}

\begin{lem}
Let
\begin{equation}
\mathcal{D}_{1}(x,v)=\mathcal{C}_{n}\frac{1}{|x-v|^{n-1}}
\end{equation}
where
\begin{equation}\mathcal{C}_{n}=\frac{1}{(n-1)\omega_{n}}.\end{equation} For
$m\geq 2$,
\begin{align}
\mathcal{D}_{m}(x,v)=\frac{\mathcal{C}_{n}}{\gamma_{1}\gamma_{2}\cdots\gamma_{m-1}}|x-v|^{2m-(n+1)}
\end{align}
if  $n$ is even, and
\begin{align}
\mathcal{D}_{m}(x,v)=\begin{cases}\frac{\mathcal{C}_{n}}{\gamma_{1}\gamma_{2}\cdots\gamma_{m-1}}|x-v|^{2m-(n+1)},\,\,\,\,m\leq \frac{n-1}{2},\vspace{2mm}\\
\frac{\mathcal{C}_{n}}{(n-1)\gamma_{1}\gamma_{2}\cdots\gamma_{\frac{n-1}{2}-1}\delta_{2}\delta_{4}\cdots\delta_{2m-n-1}}|x-v|^{2m-(n+1)}\vspace{1mm}\\
\times\left[\log|x-v|+\frac{1}{n+1}-\sum_{t=1}^{m-\frac{n+1}{2}}\left(\frac{1}{2t}+\frac{1}{2t+n-1}\right)\right],\,\,\,m\geq\frac{n+1}{2}
\end{cases}
\end{align}
if $n$ is odd, where
\begin{equation}\gamma_{k}=\delta_{2k-n+1},\,\,\,k=1,2,\ldots,m-1.\end{equation}
Then
\begin{equation}
\Delta \mathcal{D}_{1}(x,v)=0\,\,\mbox{and}\,\,\Delta
\mathcal{D}_{m}(x,v)=\mathcal{D}_{m-1}(x,v), \,\,m\geq2.
\end{equation}
\end{lem}

\begin{proof}
Using (5.2)-(5.4), it is immediate by a straightforward calculation.
\end{proof}

\begin{defn}
Let
\begin{equation}
\mathcal{K}_{m}(x,v)=
\mathcal{D}_{m}(x,v)-\mathrm{S. P.}
[\mathcal{D}_{m}](x,v)\,\,for \,\,x\neq v
\end{equation}
where
\begin{align}
\mathrm{S. P.}
[\mathcal{D}_{m}](x,v)=&\frac{c_{n}}{\gamma_{1}\gamma_{2}\cdots\gamma_{m-1}}\Big[\sum_{l=0}^{2m}P_{l}^{(\frac{n+1}{2}-m)}(x_{S^{n}}\cdot
v_{S^{n}})\\
&\times\min\left(\left|\frac{x}{v}\right|^{l}, \left|\frac{x}{v}\right|^{-l}\right)\times\max\left(|x|^{2m-n-1}, |v|^{2m-n-1}\right)\Big]\nonumber
\end{align}
for any $m$ and even $n$, or any odd $n$ with $m\leq\frac{n-1}{2};$
and
\begin{align}
\mathrm{S. P.} [\mathcal{D}_{m}](x,v)=&\frac{\mathcal{C}_{n}}{(n-1)\gamma_{1}\gamma_{2}\cdots\gamma_{\frac{n-1}{2}-1}\delta_{2}\delta_{4}\cdots\delta_{2m-n-1}}\\
&\times\Big\{\frac{1}{2}\Big[\sum_{l=0}^{2m}Q_{l}^{(\frac{n+1}{2}-m)}(x_{S^{n}}\cdot
v_{S^{n}})\times\min\left(\left|\frac{x}{v}\right|^{l}, \left|\frac{x}{v}\right|^{-l}\right)\nonumber\\
&\times\max\left(|x|^{2m-n-1}, |v|^{2m-n-1}\right)\Big]\nonumber\\
&+\left[\log\left(\max(|x|, |v|)\right)+\frac{1}{n+1}-\sum_{t=1}^{m-\frac{n+1}{2}}\left(\frac{1}{2t}+\frac{1}{2t+n-1}\right)\right]\nonumber\\
&\times\Big[\sum_{l=0}^{2m} P_{l}^{(\frac{n+1}{2}-m)}(x_{S^{n}}\cdot
v_{S^{n}})\times\min\left(\left|\frac{x}{v}\right|^{l}, \left|\frac{x}{v}\right|^{-l}\right)
\nonumber\\
&\times\max\left(|x|^{2m-n-1}, |v|^{2m-n-1}\right)\Big]\Big\}\nonumber
\end{align}
for any odd $n$ with $m\geq\frac{n+1}{2}$, where
$\delta_{s},\,\gamma_{s},\,\mathcal{C}_{n}$ are given as in (5.1) and Lemma
5.1, and the ultraspherical polynomials
$P_{l}^{(\frac{n+1}{2}-m)},\,Q_{l}^{(\frac{n+1}{2}-m)}$ are defined by
(2.16) and (2.17). Then $-\mathcal{K}_{m}(x,v)$ is said to be the
$m$th order polyharmonic fundamental solution.
\end{defn}

As Proposition 2.6, by the above definition, we have

\begin{prop}
\begin{equation}
\mathcal{K}_{m}(x,v)=\mathcal{K}_{m}(v,x)
\end{equation}
with $x\neq v$ for any $m\in \mathbb{N}$.
\end{prop}

The following theorem exhibits a nice relation between the higher
order Poisson and conjugate Poisson kernels and the higher order
polyharmonic fundamental solutions.

\begin{thm}
Let $\mathcal{K}_{m}$ and $K_{m}^{(j)}$ be as above, then
\begin{equation}
\frac{\partial}{\partial x_{j}}\mathcal{K}_{m}(x,v)=K_{m}^{(j)}(x,v)
\end{equation}
and
\begin{equation}
\frac{\partial}{\partial v_{j}}\mathcal{K}_{m}(x,v)=K_{m}^{(j)}(x,v)
\end{equation}
for any $x,v\in \mathbb{R}^{n+1}\setminus\{x\neq v\}$ and $1\leq
j\leq n+1$.
\end{thm}

\begin{proof}
By the symmetry in Proposition 5.3, it is enough to prove (5.15). To do so, at first, we claim that
\begin{equation}
\frac{\partial}{\partial x_{j}}\mathcal{D}_{m}(x,v)=D_{m}^{(j)}(x,v)
\end{equation}
for any $x,v\in \mathbb{R}^{n+1}\setminus\{x\neq v\}$ and $1\leq
j\leq n+1$.

Noting (2.5) and (5.1), we have
\begin{equation}
\delta_{s}=\frac{s}{s-2}\alpha_{s-2}
\end{equation}
for any odd $s$. To get (5.17), we consider the following three cases. \\

\noindent\textbf{Case I}: $m\geq 2$ with even $n$, or $m\leq
\frac{n-1}{2}$ with odd $n$.
\begin{align}
\frac{\partial}{\partial
x_{j}}\mathcal{D}_{m}(x,v)=&\frac{\partial}{\partial
x_{j}}\left[\frac{\mathcal{C}_{n}}{\gamma_{1}\gamma_{2}\cdots\gamma_{m-1}}|x-v|^{2m-(n+1)}\right]\\
=&\frac{(2m-n-1)\mathcal{C}_{n}}{\gamma_{1}\gamma_{2}\cdots\gamma_{m-1}}(x_{j}-v_{j})|x-v|^{2m-(n+3)}\nonumber\\
=&\frac{(2m-n-1)\mathcal{C}_{n}}{\delta_{2-(n-1)}\delta_{4-(n-1)}\cdots\delta_{2(m-1)-(n-1)}}(x_{j}-v_{j})|x-v|^{2m-(n+3)}\nonumber\\
=&\frac{(2m-n-1)\mathcal{C}_{n}}{\frac{2(m-1)-(n-1)}{-(n-1)}\alpha_{-(n-1)}\alpha_{2-(n-1)}\cdots\alpha_{2(m-2)-(n-1)}}(x_{j}-v_{j})\nonumber\\
&\times|x-v|^{2m-(n+3)}\nonumber\\
=&\frac{c_{n}}{\alpha_{2-(n+1)}\alpha_{4-(n+1)}\cdots\alpha_{2(m-1)-(n+1)}}(x_{j}-v_{j})|x-v|^{2m-(n+3)}\nonumber\\
=&\frac{c_{n}}{\beta_{1}\beta_{2}\cdots\beta_{m-1}}(x_{j}-v_{j})|x-v|^{2m-(n+3)}\nonumber\\
=&D_{m}^{(j)}(x,v)\nonumber
\end{align}
follows from (2.2), (2.12), (5.6), (5.9) and (5.18).\\

\noindent\textbf{Case II}: $m=\frac{n+1}{2}$ with odd $n$.
\begin{align}
\frac{\partial}{\partial
x_{j}}\mathcal{D}_{\frac{n+1}{2}}(x,v)=&\frac{\partial}{\partial
x_{j}}\left[\frac{\mathcal{C}_{n}}{(n-1)\gamma_{1}\gamma_{2}\cdots\gamma_{\frac{n-1}{2}-1}}\left(\log|x-v|+\frac{1}{n+1}\right)\right]\\
=&\frac{\mathcal{C}_{n}}{(n-1)\gamma_{1}\gamma_{2}\cdots\gamma_{\frac{n-1}{2}-1}}(x_{j}-v_{j})|x-v|^{-2}\nonumber\\
=&\frac{\mathcal{C}_{n}}{(n-1)\delta_{2-(n-1)}\delta_{4-(n-1)}\cdots\delta_{2(\frac{n-1}{2}-1)-(n-1)}}(x_{j}-v_{j})\nonumber\\
&\times|x-v|^{-2}\nonumber\\
=&\frac{\mathcal{C}_{n}}{(n-1)\left[\frac{2(\frac{n-1}{2}-1)-(n-1)}{-(n-1)}\alpha_{-(n-1)}\alpha_{2-(n-1)}\cdots\alpha_{2(\frac{n-1}{2}-2)-(n-1)}\right]}\nonumber\\
&\times(x_{j}-v_{j})|x-v|^{-2}\nonumber\\
=&\frac{\mathcal{C}_{n}}{\frac{1}{-(n-1)}\alpha_{2-(n+1)}\alpha_{4-(n+1)}\cdots\alpha_{2(\frac{n-1}{2}-1)-(n+1)}\alpha_{2(\frac{n+1}{2}-1)-(n+1)}}\nonumber\\
&\times(x_{j}-v_{j})|x-v|^{-2}\nonumber\\
=&\frac{c_{n}}{\beta_{1}\beta_{2}\cdots\beta_{\frac{n+1}{2}-1}}(x_{j}-v_{j})|x-v|^{-2}\nonumber\\
=&D_{\frac{n+1}{2}}^{(j)}(x,v)\nonumber
\end{align}
follows from (2.2), (2.12), (5.1), (5.6), (5.9) and (5.18).\\

\noindent\textbf{Case III}: $m\geq\frac{n+3}{2}$ with odd $n$.

\begin{align}
\frac{\partial}{\partial x_{j}}\mathcal{D}_{m}(x,v)=&\frac{\partial}{\partial x_{j}}\Big\{\frac{\mathcal{C}_{n}}
{(n-1)\gamma_{1}\gamma_{2}\cdots\gamma_{\frac{n-1}{2}-1}\delta_{2}\delta_{4}\cdots\delta_{2m-n-1}}|x-v|^{2m-(n+1)}\\
&\times\Big[\log|x-v|+\frac{1}{n+1}-\sum_{t=1}^{m-\frac{n+1}{2}}\Big(\frac{1}{2t}+\frac{1}{2t+n-1}\Big)\Big]\Big\}\nonumber\\
=&\frac{(2m-n-1)\mathcal{C}_{n}}{(n-1)\gamma_{1}\gamma_{2}\cdots\gamma_{\frac{n-1}{2}-1}\delta_{2}\delta_{4}\cdots\delta_{2m-n-1}}(x_{j}-v_{j})|x-v|^{2m-(n+3)}\nonumber\\
&\times\Big[\log|x-v|+\frac{1}{n+1}-\sum_{t=1}^{m-\frac{n+1}{2}}\Big(\frac{1}{2t}+\frac{1}{2t+n-1}\Big)\Big]\nonumber\\
&+\frac{\mathcal{C}_{n}}{(n-1)\gamma_{1}\gamma_{2}\cdots\gamma_{\frac{n-1}{2}-1}\delta_{2}\delta_{4}\cdots\delta_{2m-n-1}}(x_{j}-v_{j})|x-v|^{2m-(n+3)}\nonumber\\
=&\frac{(2m-n-1)\mathcal{C}_{n}}{(n-1)\delta_{2-(n-1)}\delta_{4-(n-1)}\cdots\delta_{2(\frac{n-1}{2}-1)-(n-1)}\delta_{2}\delta_{4}\cdots\delta_{2m-n-1}}\nonumber\\
&\times(x_{j}-v_{j})|x-v|^{2m-(n+3)}\Big[\log|x-v|+\frac{1}{n+1}\nonumber\\
&-\sum_{t=1}^{m-\frac{n+1}{2}}\Big(\frac{1}{2t}+\frac{1}{2t+n-1}\Big)\Big]\nonumber\\
&+\frac{\mathcal{C}_{n}}{(n-1)\delta_{2-(n-1)}\delta_{4-(n-1)}\cdots\delta_{2(\frac{n-1}{2}-1)-(n-1)}\delta_{2}\delta_{4}\cdots\delta_{2m-n-1}}\nonumber\\
&\times(x_{j}-v_{j})|x-v|^{2m-(n+3)}\nonumber\\
=&\frac{c_{n}}{(n+1)\beta_{1}\beta_{2}\cdots\beta_{\frac{n+1}{2}-1}\alpha_{2}\alpha_{4}\cdots\alpha_{2m-n-3}}(x_{j}-v_{j})|x-v|^{2m-(n+3)}\nonumber\\
&\times\Big[\log|x-v|+\frac{1}{n+1}-\sum_{t=1}^{m-\frac{n+1}{2}}\Big(\frac{1}{2t}+\frac{1}{2t+n-1}\Big)\Big]\nonumber\\
&+\frac{1}{2m-n-1}\frac{c_{n}}{(n+1)\beta_{1}\beta_{2}\cdots\beta_{\frac{n+1}{2}-1}\alpha_{2}\alpha_{4}\cdots\alpha_{2m-n-3}}\nonumber\\
&\times(x_{j}-v_{j})|x-v|^{2m-(n+3)}\nonumber\\
=&\frac{c_{n}}{(n+1)\beta_{1}\beta_{2}\cdots\beta_{\frac{n+1}{2}-1}\alpha_{2}\alpha_{4}\cdots\alpha_{2m-n-3}}(x_{j}-v_{j})|x-v|^{2m-(n+3)}\nonumber\\
&\times\left[\log|x-v|-\sum_{t=1}^{m-\frac{n+3}{2}}\left(\frac{1}{2t}+\frac{1}{2t+n+1}\right)\right]\nonumber\\
=&D_{m}^{(j)}(x,v)\nonumber
\end{align}
follows from (2.2), (2.12), (5.1), (5.6), (5.9) and (5.18), where the fourth equality is based on the following calculations (by repeatedly invoking (5.18)):
\begin{align}
&(n-1)\delta_{2-(n-1)}\delta_{4-(n-1)}\cdots\delta_{2(\frac{n-1}{2}-1)-(n-1)}\delta_{2}\delta_{4}\cdots\delta_{2m-n-1}
\end{align}
\begin{align}
=&2(n-1)(n+1)\prod_{k=0}^{\frac{n-1}{2}-2}\left[\frac{2(k+1)-(n-1)}{2k-(n-1)}\alpha_{2k-(n-1)}\right]\times\prod_{l=1}^{m-\frac{n+3}{2}}\left[\frac{2l+2}{2l}\alpha_{2l}\right]\nonumber\\
=&\frac{2m-n-1}{1-n}\left\{(n+1)\prod_{k=1}^{\frac{n-1}{2}-1}\alpha_{2k-(n+1)}\times[-2(n-1)]\times\prod_{l=1}^{m-\frac{n+3}{2}}\alpha_{2l}\right\}\nonumber\\
=&\frac{2m-n-1}{1-n}\left\{(n+1)\prod_{k=1}^{\frac{n-1}{2}-1}\alpha_{2k-(n+1)}\times[(-2)(-2+n+1)]\times\prod_{l=1}^{m-\frac{n+3}{2}}\alpha_{2l}\right\}\nonumber\\
=&\frac{2m-n-1}{1-n}\left\{(n+1)\prod_{k=1}^{\frac{n+1}{2}-1}\alpha_{2k-(n+1)}\times\prod_{l=1}^{m-\frac{n+3}{2}}\alpha_{2l}\right\}\nonumber\\
=&\frac{2m-n-1}{1-n}(n+1)\beta_{1}\beta_{2}\cdots\beta_{\frac{n+1}{2}-1}\alpha_{2}\alpha_{4}\cdots\alpha_{2m-n-3}\nonumber
\end{align}
in which $-2(n-1)=(-2)\left(-2+(n+1)\right)=\alpha_{2(\frac{n+1}{2}-1)-(n+1)}=\beta_{\frac{n+1}{2}-1}$ that has been already used in the fifth equality of (5.20).

By (5.17), we have
\begin{equation}
\frac{\partial}{\partial
x_{j}}\mathcal{K}_{m}(x,v)-K_{m}^{(j)}(x,v)=\mathrm{S. P.}
[D_{m}^{(j)}](x,v)-\frac{\partial}{\partial x_{j}}\mathrm{S. P.}
[\mathcal{D}_{m}](x,v)
\end{equation}
for any $x,v\in \mathbb{R}^{n+1}$ with $x\neq v$ and sufficiently
large $|v|$ (in fact, for any $|v|$).  By Definition 2.3, $\frac{\partial}{\partial
x_{j}}\mathcal{K}_{m}(x,v)-K_{m}^{(j)}(x,v)=\mathrm{S. P.}
[D_{m}^{(j)}](x,v)-\frac{\partial}{\partial x_{j}}\mathrm{S. P.}
[\mathcal{D}_{m}](x,v)=0$. Then (5.15) follows and the theorem is
completed.
\end{proof}

\begin{rem}
In the proofs of the above theorem and Theorem 3.5, we respectively obtain that
\begin{equation}
\mathrm{S. P.}
[D_{m}^{(j)}](x,v)=\frac{\partial}{\partial x_{j}}\mathrm{S. P.}
[\mathcal{D}_{m}](x,v)
\end{equation}
and
\begin{equation}
\mathrm{S. P.}
[D_{m-1}^{(j)}](x,v)=\Delta\mathrm{S. P.}
[D_{m}^{(j)}](x,v).
\end{equation}
From these identities, it is easy to find some identities on the ultraspherical polynomials $P_{l}^{(\lambda)}$ and $Q_{l}^{(\lambda)}$. However, we will not want to pursue these results in this article.
\end{rem}

\section{Polyharmonic Neumann problems in Lipschitz graph domains}

In this section, we will consider the polyharmonic Neumann problems
(1.2) in Lipschitz graph domains as follows
\begin{equation}\begin{cases}\Delta^{m}u=0,\,\,in\,\, D,\vspace{2mm}\\
  \frac{\partial}{\partial N}\Delta^{j}u=g_{j},\,\,on\,\,\partial
  D,
  \end{cases}
  \end{equation}
where $\nabla (u-\mathcal{M}_{1}\widetilde{g}_{0})\in L^{p}(D)$ with $\|\nabla (u-\mathcal{M}_{1}\widetilde{g}_{0})\|_{L^{p}(D)}\leq C\sum_{j=1}^{m-1}\|g_{j}\|_{L^{p}(\partial D, wd\sigma)}$, the Laplacian $\Delta=\sum_{k=1}^{n+1}\frac{\partial^{2}}{\partial
x_{k}^{2}}$, the gradient operator $\nabla=\left(\frac{\partial}{\partial x_{1}}, \frac{\partial}{\partial x_{2}},\ldots, \frac{\partial}{\partial x_{n+1}}\right)$,   $D$ is a Lipschitz graph domain stated as in Theorem 3.5, $g_{0}\in
L^{p}(\partial D)$, $g_{j}\in
L^{p}(\partial D, wd\sigma)$ for some suitable $p>1$, the $(p, 2m-2,\frac{3}{2})$ weight $w$ on $\partial D$ is given in Section 3.1,
$\frac{\partial}{\partial N}$ denotes the outward normal derivative, $\widetilde{g}_{0}$ ia related to all the boundary data $g_{j}$, $0\leq j<m$ and $m\in \mathbb{N}$.

\begin{defn}Let $D$ be a Lipschitz domain in $\mathbb{R}^{n+1}$ with the
boundary $\partial D$. Set
\begin{equation}
\mathcal{M}_{j}f(X)=\int_{\partial D} \mathcal{K}_{j}(X,Q)
f(Q)d\sigma(Q), \,\,X\in D,
\end{equation}
where $1\leq j< \infty$, $\mathcal{K}_{j}$ is the $j$th order
polyharmonic fundamental solution, $d\sigma$ is the surface measure
on $\partial D$, and $f\in L^{p}(\partial D)$ for some suitable $p$.
$\mathcal{M}_{j}f$ is called the $j$th-layer $\mathcal{S}$-potential
of $f$.
\end{defn}

\begin{rem} It is well known that $-\mathcal{K}_{1}$ is the fundamental
solution of the Laplacian and $\mathcal{M}_{1}$ is the classical
single layer potential.\end{rem}

By the properties of polyharmonic fundamental solutions, we have
\begin{thm}
Let $\{\mathcal{K}_{m}\}_{m=1}^{\infty}$ be the sequence of the
polyharmonic fundamental solutions, and $D$ be a Lipschitz graph
domain in $\mathbb{R}^{n+1}$ with Lipschitz graph boundary $\partial D$, which is the same as in Theorem 3.5, then
\begin{enumerate}
  \item For all $m\in \mathbb{N}$, $\mathcal{K}_{m}\in (C^{\infty}\times
        C)(D\times \partial D)$, the non-tangential
        boundary value \begin{equation*}
        \lim_{\substack {X\rightarrow P \\X\in \Gamma_{\gamma}(P),\, Q\in
        \partial D}}\mathcal{K}_{m}(X,Q)=\mathcal{K}_{m}(P,Q)
        \end{equation*}
        exists for all $P\in \partial D$ and $ P\ne Q\in
        \partial D$; $\mathcal{K}_{m}(\cdot,P)$ can be continuously extended to
        $\overline{D}\setminus\{P\}$ for any fixed $P\in
        \partial D$;
  \item For  $m\geq 2$,
        \begin{equation*}
        |\mathcal{K}_{m}(X,Q)|\leq M\frac{1}{\left(1+|Q|^{2}\right)^{\frac{n+1+\epsilon}{2}}}
        \end{equation*}
        for any $(X,Q)\in D_{c}\times\{Q\in\partial D: |Q|>T\}$, where $0<\epsilon<1$,
        $D_{c}$ is any compact subset of $\overline{D}$,
        $T$ is a sufficiently large positive real number and $M$ denotes
        some positive constant depending only on $\epsilon$, $D_{c}$ and $T$;
  \item $\Delta_{X} \mathcal{K}_{1}(X, Y)=\Delta_{Y} \mathcal{K}_{1}(X, Y)=0$ and $\Delta_{X} \mathcal{K}_{m}(X, Y)=\Delta_{Y} \mathcal{K}_{m}(X, Y)=\mathcal{K}_{m-1}(X, Y)$ for $m>1$, $X, Y\in \mathbb{R}^{n+1}\setminus\{0\}$ and $X\neq Y$, where $\Delta_{X}=\sum_{j=1}^{n+1}\frac{\partial}{\partial x_{j}}$ and $\Delta_{Y}=\sum_{j=1}^{n+1}\frac{\partial}{\partial y_{j}}$;
  \item The non-tangential limit \begin{equation}\lim_{\substack{X\rightarrow P\\ X\in \Gamma_{\gamma}(P)}}\left\langle\nabla\left(\int_{\partial D}\mathcal{K}_{1}(X, Q)f(Q)d\sigma(Q)\right),\,\, n_{P}\right\rangle
        =-\frac{1}{2}f(P)+T^{*}f(P),\end{equation}
        for any $f\in L^{p}(\partial D)$, $1\leq p<\infty$;
  \item The non-tangential limit \begin{equation}\lim_{\substack{X\rightarrow P\\ X\in \Gamma_{\gamma}(P)}}\left\langle\nabla\left(\int_{\partial D}\mathcal{K}_{m}(X, Q)f(Q)d\sigma(Q)\right),\,\, n_{P}\right\rangle
        =-\mathrm{K}^{*}_{m}f(P),\end{equation}
        for any $m\geq2$ and $f\in L^{p}(\partial D)$, $1\leq p\leq\infty$, where \begin{equation}
        \mathrm{K}^{*}_{m}f(P)=\int_{\partial D}\langle K_{m}(Q, P),
        n_{P}\rangle f(Q)d\sigma(Q)
        \end{equation}
        which is the adjoint operator of $\mathrm{K}_{m}$.
        \end{enumerate}
\end{thm}

\begin{proof}
It is similar to Theorem 3.5 by invoking Lemma 5.1 and Theorem 5.4.
\end{proof}

\begin{rem}
The operator $\mathrm{K}^{*}_{m}$ has the same boundedness properties as the operator $\mathrm{K}_{m}$ does. For instance, it is also bounded form $L^{p}(\partial D, wd\sigma)$ to $L^{p}(\partial D)$ for any $w\in W^{p,2m-2,\frac{1}{2}}(\partial D)$ and $1\leq p\leq\infty$. The details can be seen in the following Theorem 6.8 in Section 6.1.
\end{rem}

\begin{thm}
Let $\{\,\mathcal{K}_{m}\,\}_{m=1}^{\infty}$ be the sequence of the
polyharmonic fundamental solutions, and $E$ be a simply connected
unbounded domain in $\mathbb{R}^{n+1}$ with smooth boundless
boundary $\partial E$. Then for any $m>1$ and $f\in L^{p}(\partial
E)$, $p\geq1$,
\begin{equation}
\Delta\left(\int_{\partial E}\mathcal{K}_{m}(X, Q)
f(Q)d\sigma(Q)\right)=\int_{\partial E}\mathcal{K}_{m-1}(X,
Q)f(Q)d\sigma(Q),
\end{equation}
where $X\in \mathbb{R}^{n+1}\setminus \partial E$, namely,
\begin{equation}
\Delta \mathcal{M}_{m}f(X)=\mathcal{M}_{m-1}f(X),\,\, X\in
\mathbb{R}^{n+1}\setminus
\partial E.
\end{equation}
\end{thm}

\begin{proof}
It is similar to Theorem 4.3 by using the analogues of Lemma 4.1, Corollary 4.2 and the claim (3) in the last theorem.
\end{proof}

\begin{rem}
As Remark 4.4 stated, the above theorem also holds in the case of replacing the smooth domain $E$ by the Lipschitz graph domain $D$ given in Theorem 3.5.
\end{rem}

By the last two theorems, Lemma 3.4 and the results in the following Section 6.1, we
can solve the polyharmonic Neumann problems in Lipschitz domains as follows.
\begin{thm}
Let $\{\,\mathcal{K}_{m}\,\}_{m=1}^{\infty}$ be the sequence of the
polyharmonic fundamental solutions, and $D$ be a Lipschitz graph
domain in $\mathbb{R}^{n+1}$ with Lipschitz graph boundary $\partial
D$ as in Theorem 3.5, then for any $m>1$, there exists $\varepsilon=\varepsilon(D)>0$
such that the PHN problem (1.2) with the data $g_{0}\in
L^{p}(\partial D)$, $g_{j}\in
L^{p}(\partial D, wd\sigma)$ with $w\in \mathcal{W}^{p, 2m-2,\frac{3}{2}}(\partial D)$, $1\leq j<m$, $1<p<2+\varepsilon$, is
solvable and a solution is given by
\begin{align}
u(X)&=\sum_{j=1}^{m}\int_{\partial D}\mathcal{K}_{j}(X,
Q)\widetilde{g}_{j-1}(Q)d\sigma(Q),\\
&=\sum_{j=1}^{m}\mathcal{M}_{j}\widetilde{g}_{j-1}(X),\,\, X\in D,\nonumber
\end{align}
where
\begin{equation}\widetilde{g}_{m-1}=\left(-\frac{1}{2}I+T^{*}\right)^{-1}g_{m-1}\end{equation}
and
\begin{equation}\widetilde{g}_{l}=\left(-\frac{1}{2}I+T^{*}\right)^{-1}\left(g_{l}+\sum_{j=l+2}^{m}\mathrm{K}^{*}_{j-l}\widetilde{g}_{j-1}\right)\end{equation}
with $0\leq l\leq m-2$, which satisfying the following estimate
\begin{equation}
\|\nabla (u-\mathcal{M}_{1}\widetilde{g}_{0})\|_{L^{p}(D)}\leq C\sum_{j=1}^{m-1}\|g_{j}\|_{L^{p}(\partial D, wd\sigma)}.
\end{equation}
Under this estimate, the solution (6.8) with (6.9) and (6.10) is unique up to a constant.
Furthermore, if $w\in \mathcal{W}^{p, 2m-1,\frac{3}{2}}(\partial D)\,\big(\!\!\subset \mathcal{W}^{p, 2m-2,\frac{3}{2}}(\partial D)\big)$, then the above solution also satisfies the following estimate
\begin{equation}
\|u-\mathcal{M}_{1}\widetilde{g}_{0}\|_{L^{p}(D)}\leq C\sum_{j=1}^{m-1}\|g_{j}\|_{L^{p}(\partial D, wd\sigma)},
\end{equation}
and is unique under the last estimate.
\end{thm}

\begin{proof}
It is similar to Theorem 4.5 by lemma 3.4, and Theorems 6.3, 6.8, 6.11 and 6.12 below.
\end{proof}

\subsection{$L^{p}$ boundedness properties of operators
$\mathrm{K}^{*}_{m}$ and multi-layer $\mathcal{S}$-potentials $\mathcal{M}_{j}$ and their gradients}

In this section, we study the $L^{p}$ boundedness properties of the operators $\mathrm{K}^{*}_{m}$ given in (6.5) and the multi-layer $\mathcal{S}$-potentials $\mathcal{M}_{j}$ defined by (6.2) and their gradients, which are very significant for the solving program to the PHN and PHR problems (i.e., (1.2) and (1.3)) in this paper. More precisely, we have

\begin{thm}
Let the Lipschitz domain $D$ and the operators $\mathrm{K}^{*}_{m}$, $m\geq 2$, be the same as before, $w\in \mathcal{W}^{p,2m-2,\frac{1}{2}}(\partial D)$, $1\leq p<\infty$, then
\begin{equation}\|\mathrm{K}^{*}_{m}f\|_{L^{p}(\partial D)}\leq
C\| f \|_{L^{p}(\partial D, wd\sigma)}\end{equation}
for any $f\in L^{p}(\partial D, wd\sigma)$, where $C$ is a constant depending only on $m, n, p$ and $d_{0}={\rm dist}(0, \partial D)$. That is, $\mathrm{K}^{*}_{m}$, $m\geq 2$, are bounded from $L^{p}(\partial D, wd\sigma)$ to $L^{p}(\partial D)$ for any $w\in \mathcal{W}^{p,2m-2,\frac{1}{2}}(\partial D)$ with $1\leq p<\infty$.
\end{thm}

\begin{proof}
It is similar to the argument of Theorem 3.11.
\end{proof}

\begin{thm}
Let the Lipschitz domain $D$ and operators $\mathcal{M}_{j}$, $j\geq 2$, be the same as before, $w\in \mathcal{W}^{p,2j-1,\frac{1}{2}}(\partial D)$, $1\leq p<\infty$, then
\begin{equation}\|\mathcal{M}_{j}f\|_{L^{p}(\partial D)}\leq
C\| f \|_{L^{p}(\partial D, wd\sigma)}\end{equation}
for any $f\in L^{p}(\partial D, wd\sigma)$, where $C$ is a constant depending only on $m, n, p$ and $d_{0}$. That is, $\mathcal{M}_{j}$, $j\geq 2$, are bounded from $L^{p}(\partial D, wd\sigma)$ to $L^{p}(\partial D)$ for any $w\in \mathcal{W}^{p,2j-1,\frac{1}{2}}(\partial D)$ with $1\leq p<\infty$.
\end{thm}

\begin{proof}
It is similar to Theorem 3.11.
\end{proof}

\begin{thm}
Let the Lipschitz domain $D$ and operators $\mathcal{M}_{j}$, $j\geq 2$, be the same as before, $w\in \mathcal{W}^{p,2j-2,\frac{1}{2}}(\partial D)$, $1\leq p<\infty$, then
\begin{equation}\|\nabla\mathcal{M}_{j}f\|_{L^{p}(\partial D)}\leq
C\| f \|_{L^{p}(\partial D, wd\sigma)}\end{equation}
for any $f\in L^{p}(\partial D, wd\sigma)$, where $C$ is a constant depending only on $m, n, p$ and $d_{0}$. That is, $\nabla\mathcal{M}_{j}$, $j\geq 2$, are bounded from $L^{p}(\partial D, wd\sigma)$ to $L^{p}(\partial D)$ for any $w\in \mathcal{W}^{p,2j-2,\frac{1}{2}}(\partial D)$ with $1\leq p<\infty$.
\end{thm}

\begin{proof}
It is similar to the argument of Theorem 3.11 by using Theorem 5.4.
\end{proof}

\begin{thm}
Let the Lipschitz domain $D$ and operators $\mathcal{M}_{j}$, $j\geq 2$, be the same as before, $w\in \mathcal{W}^{p,2j-1,\frac{3}{2}}(\partial D)$, $1\leq p<\infty$, then
\begin{equation}\|\mathcal{M}_{j}f\|_{L^{p}(D)}\leq
C\| f \|_{L^{p}(\partial D, wd\sigma)}\end{equation}
for any $f\in L^{p}(\partial D, wd\sigma)$, where $C$ is a constant depending only on $m, n, p$ and $d_{0}$. That is, $\mathcal{M}_{j}$, $j\geq 2$, are bounded from $L^{p}(\partial D, wd\sigma)$ to $L^{p}(D)$ for any $w\in \mathcal{W}^{p,2j-1,\frac{3}{2}}(\partial D)$ with $1\leq p<\infty$.
\end{thm}

\begin{proof}
It is similar to Theorem 3.12.
\end{proof}

\begin{thm}
Let the Lipschitz domain $D$ and operators $\mathcal{M}_{j}$, $j\geq 2$, be the same as before, $w\in \mathcal{W}^{p,2j-2,\frac{3}{2}}(\partial D)$, $1\leq p<\infty$, then
\begin{equation}\|\nabla\mathcal{M}_{j}f\|_{L^{p}(D)}\leq
C\| f \|_{L^{p}(\partial D, wd\sigma)}\end{equation}
for any $f\in L^{p}(\partial D, wd\sigma)$, where $C$ is a constant depending only on $m, n, p$ and $d_{0}$. That is, $\nabla\mathcal{M}_{j}$, $j\geq 2$, are bounded from $L^{p}(\partial D, wd\sigma)$ to $L^{p}(D)$ for any $w\in \mathcal{W}^{p,2j-2,\frac{3}{2}}(\partial D)$ with $1\leq p<\infty$.
\end{thm}

\begin{proof}
It is similar to the argument of Theorem 3.12 by invoking Theorem 5.4.
\end{proof}

\section{Regularity of polyharmonic Dirichlet problems in Lipschitz graph domains}

In this section, we will consider the polyharmonic regularity problems
(1.3) in Lipschitz domains as follows
\begin{equation}\begin{cases}\Delta^{m}u=0,\,\,in\,\, D,\vspace{2mm}\\
  \Delta^{j}u=h_{j},\,\,on\,\,\partial
  D,
  \end{cases}
  \end{equation}
where $\nabla (u-\mathcal{M}_{1}\widetilde{h}_{0})\in L^{p}(D)$ with $\|\nabla (u-\mathcal{M}_{1}\widetilde{h}_{0})\|_{L^{p}(D)}\leq C\sum_{j=1}^{m-1}\|h_{j}\|_{L_{1}^{p}(\partial D, wd\sigma)}$, the Laplacian $\Delta=\sum_{k=1}^{n+1}\frac{\partial^{2}}{\partial
x_{k}^{2}}$, the gradient operator $\nabla=\left(\frac{\partial}{\partial x_{1}}, \frac{\partial}{\partial x_{2}},\ldots, \frac{\partial}{\partial x_{n+1}}\right)$,   $D$ is a Lipschitz graph domain stated as Theorem 3.5, $h_{0}\in
L_{1}^{p}(\partial D)$, $h_{j}\in
L_{1}^{p}(\partial D, wd\sigma)$ for some suitable $p>1$, the $(p, 2m-1,\frac{3}{2})$ weight $w$ on $\partial D$ is given in Section 3.1, $\widetilde{h}_{0}$ is related to all the boundary date $h_{j}$, $0\leq j<m$ and $m\in \mathbb{N}$.

Once more, due to Dahlberg, Kenig and Verchota et al., we have

\begin{lem}[\cite{dk1,v0}]
There exists $\varepsilon=\varepsilon(D)>0$ such that
$\mathcal{M}_{1}$ is an invertible mapping from
$L^{p}(\partial D)$ onto $L_{1}^{p}(\partial D)$, $1<p<2+\varepsilon$, where $L_{1}^{p}(\partial D)=\{f\in L^{p}(\partial D): \nabla_{T}f \,\,exist \,\,a.\,e.\,\,on\,\,\partial D,\,\, and\,\, |\nabla_{T}f|\in L^{p}(\partial D)\}$ with the norm $\|f\|_{L_{1}^{p}(\partial D)}=\|f\|_{L^{p}(\partial D)}+\|\nabla_{T}f\|_{L^{p}(\partial D)}$ in which $\nabla_{T}$ is the tangential gradient.
\end{lem}

\begin{thm}
Let $\{\,\mathcal{K}_{m}\,\}_{m=1}^{\infty}$ be the sequence of the
polyharmonic fundamental solutions, and $D$ be a Lipschitz graph
domain in $\mathbb{R}^{n+1}$ with Lipschitz graph boundary $\partial
D$ as in Theorem 3.5, then for any $m>1$, there exists $\varepsilon=\varepsilon(D)>0$
such that the PHR problem (1.3) with the data $h_{0}\in
L_{1}^{p}(\partial D)$, $h_{j}\in
L_{1}^{p}(\partial D, wd\sigma)$ with $w\in \mathcal{W}^{p,2m-1,\frac{3}{2}}(\partial D)$, $1\leq j<m$, $1<p<2+\varepsilon$, is
solvable and a solution is given by
\begin{align}
u(X)&=\sum_{j=1}^{m}\int_{\partial D}\mathcal{K}_{j}(X,
Q)\widetilde{h}_{j-1}(Q)d\sigma(Q),\\
&=\sum_{j=1}^{m}\mathcal{M}_{j}\widetilde{h}_{j-1}(X),\,\, X\in D,\nonumber
\end{align}
where
\begin{equation}\widetilde{h}_{m-1}=\mathcal{M}_{1}^{-1}h_{m-1}\end{equation}
and
\begin{equation}\widetilde{h}_{l}=\mathcal{M}_{1}^{-1}\left(h_{l}-\sum_{j=l+2}^{m}\mathcal{M}_{j-l}\widetilde{h}_{j-1}\right)\end{equation}
with $0\leq l\leq m-2$, which satisfying the following estimate
\begin{equation}
\|\nabla (u-\mathcal{M}_{1}\widetilde{h}_{0})\|_{L^{p}(D)}\leq C\sum_{j=1}^{m-1}\|h_{j}\|_{L_{1}^{p}(\partial D, wd\sigma)}.
\end{equation}
Under this estimate, the solution (7.2) with (7.3) and (7.4) is unique up to a constant.
Furthermore, the above solution also satisfies the following estimate
\begin{equation}
\|u-\mathcal{M}_{1}\widetilde{h}_{0}\|_{L^{p}(D)}\leq C\sum_{j=1}^{m-1}\|h_{j}\|_{L_{1}^{p}(\partial D, wd\sigma)}.
\end{equation}
and is unique under the last estimate.

\begin{proof}
It is similar to Theorem 4.5 by using Lemma 7.1, Theorems 6.9-6.12 and 7.3 below.
\end{proof}
\end{thm}

\subsection{Regularity of multi-layer $\mathcal{S}$-potentials $\mathcal{M}_{j}$}

In this section, we study the regularity of the multi-layer $\mathcal{S}$-potentials $\mathcal{M}_{j}$, which are very significant for the solving program to the PHR problems (1.3) in this paper. More precisely, we have

\begin{thm}
Let the Lipschitz domain $D$ and operators $\mathcal{M}_{j}$, $j\geq 2$, be the same as before, $w\in \mathcal{W}^{p,2j-2,\frac{1}{2}}(\partial D)$, $1\leq p<\infty$, then
\begin{equation}\|\nabla_{T}\mathcal{M}_{j}f\|_{L^{p}(\partial D)}\leq
C\| f \|_{L^{p}(\partial D, wd\sigma)}\end{equation}
for any $f\in L^{p}(\partial D, wd\sigma)$, where $\nabla_{T}$ denotes the tangential gradient, $C$ is a constant depending only on $m, n, p$ and $d_{0}$. So $\mathcal{M}_{j}$, $j\geq 2$, are bounded from $L^{p}(\partial D, wd\sigma)$ to $L_{1}^{p}(\partial D)$ for any $w\in \mathcal{W}^{p,2j-1,\frac{1}{2}}(\partial D)$ with $1\leq p<\infty$.
\end{thm}

\begin{proof}
It is similar to the argument of Theorem 3.11, or directly follows from Theorems 6.8 and 6.10 by the following fact
\begin{align}
\|\nabla_{T}\mathcal{M}_{j}f\|_{L^{p}(\partial D)}&=\left\|\nabla\mathcal{M}_{j}f-\left(\frac{\partial}{\partial N}\mathcal{M}_{j}f\right)\cdot n\right\|_{L^{p}(\partial D)}\\
&\leq 2^{p-1}\left(\|\nabla\mathcal{M}_{j}f\|_{L^{p}(\partial D)}+\left\|\frac{\partial}{\partial N}\mathcal{M}_{j}f\right\|_{L^{p}(\partial D)}\right)\nonumber\\
&= 2^{p-1}\left(\|\nabla\mathcal{M}_{j}f\|_{L^{p}(\partial D)}+\left\|\mathrm{K}_{j}^{*}f\right\|_{L^{p}(\partial D)}\right)\nonumber\\
&\leq C\|f\|_{L^{p}(\partial D, wd\sigma)}\nonumber
\end{align}
since $\nabla\mathcal{M}_{j}f=\nabla_{T}\mathcal{M}_{j}f\oplus\left(\frac{\partial}{\partial N}\mathcal{M}_{j}f\right)\cdot n$, where $\oplus$ denotes the operation of direct sum, and $n$ is the outward unit normal vector.
\end{proof}

\begin{rem}
It must be noted, using the facts in Remark 3.10, that all the results in Sections 3.1-7 hold with the weight spaces $\mathcal{W}^{p, l,\frac{1}{2}}(\partial D)$ and $\mathcal{W}^{p, l,\frac{3}{2}}(\partial D)$ replaced by $\mathcal{W}^{p, l,\epsilon}(\partial D)$ and $\mathcal{W}^{p, l,1+\epsilon}(\partial D)$ for any $0<\epsilon<1$.
\end{rem}

\section{Bounded Lipschitz domains}

In this section, we mainly consider the corresponding polyharmonic Dirichlet, Neumann, and regularity problems in $L^{p}$ in bounded Lipschitz domains. Now the higher order conjugate Poisson and Poisson kernels $K_{m}^{(j)}=D_{m}^{(j)}$, and the polyharmonic fundamental solutions $\mathcal{K}_{m}=\mathcal{D}_{m}$, $1\leq j\leq n+1$, $m\in \mathbb{N}$. In other words, here $\mathrm{S. P.} [K_{m}^{(j)}]\equiv 0$ and $\mathrm{S. P.} [\mathcal{K}_{m}]\equiv 0$ for any $1\leq j\leq n+1$ and $m\in \mathbb{N}$.

In the same way, due to Dahlberg, Kenig and Verchota et al., we have
\begin{lem}[\cite{dk1,v0}]
There exists $\varepsilon=\varepsilon(D)>0$ such that
$\frac{1}{2}I-T^{*}$ is an invertible mapping from
$L^{p}_{0}(\partial D)$ onto $L_{0}^{p}(\partial D)$, $1<p<2+\varepsilon$, where $L_{0}^{p}(\partial D)=\{f\in L^{p}(\partial D): \int_{\partial D}fd\sigma=0\}$.
\end{lem}

As some preliminaries, we firstly establish some lemmas as follows.

\begin{lem}
Let $D$ be a bounded Lipschitz domain, $D_{m}=(D_{m}^{(1)}, \ldots, D_{m}^{(n+1)})$ in which $D_{m}^{(j)}$ are defined as in Lemma 2.2, then there exists a constant $C=C(m,n,D)$ such that
\begin{equation}
\sup_{Q\in\partial D}\left(\int_{\partial D}\big|\langle D_{m}(Q, P), n_{P}\rangle\big| d\sigma(P)\right)<C
\end{equation}
and
\begin{equation}
\sup_{Q\in\partial D}\left(\int_{\partial D}\big|\langle D_{m}(Q, P), n_{Q}\rangle\big| d\sigma(P)\right)<C
\end{equation}
for any $m\geq 2$, where $n_{P}$ and $n_{Q}$ are the unit outward normal vectors respectively at $P$ and $Q$ on $\partial D$.
\end{lem}

\begin{proof}
At first, we observe that
\begin{equation}
|\langle D_{m}(Q, P), n_{\cdot}\rangle|\leq C_{m,n}|P-Q|^{2m-(n+2)}\left(1+\big|\log|P-Q|\big|\right).
\end{equation}
So it is sufficient to verify (8.1). By the definition of bounded Lipschitz domain, set $\{L_{1},\ldots, L_{s}\}$ be a finite cover of circular coordinate cylinders on $\partial D$ centered respectively at $Q_{j}$, $1\leq j\leq s$ whose bases have positive distances from $\partial D$. That is, there exists a Lipschitz function $\varphi_{j}: \mathbb{R}^{n}\rightarrow \mathbb{R}$, $1\leq j\leq s$ such that
\begin{description}
  \item[(i)] $|\varphi_{j}(\underline{x})-\varphi_{j}(\underline{y})|\leq \mathcal{L}_{j}|\underline{x}-\underline{y}|$ for any $\underline{x}, \underline{y}\in \mathbb{R}^{n}$ with $0<\mathcal{L}_{j}<\infty$;
  \item[(ii)] $L_{j}\cap D=\{(\underline{x}, x_{n+1}): x_{n+1}>\varphi_{j}(\underline{x})\}$;
  \item[(iii)] $L\cap \partial D=\{(\underline{x}, x_{n+1}): x_{n+1}=\varphi_{j}(\underline{x})\}$;
  \item[(iv)] $Q_{j}=(\underline{0}, \varphi_{j}(\underline{0}))$,
\end{description}
where $\underline{x}=\{x_{1}, \ldots, x_{n}\}\in \mathbb{R}^{n}$. Let $\mathcal{L}=\max_{1\leq j\leq s}\mathcal{L}_{j}$, $\mathcal{L}$ is usually called the Lipschitz constant (or Lipschitz character). By a rearrangement, we can assume that all $L_{j}$ are adjacent with each other in turn.

Denote that $d_{j}=\mathrm{dist}\{Q_{j}, \partial (L_{j}\cap\partial D)\}$, $1\leq j\leq s$. In the coordinate system associated with $(L_{j}, Q_{j})$, define the projection $\pi_{j}: \mathbb{R}^{n+1}\rightarrow \mathbb{R}^{n}$ with $\pi_{j}(\underline{x}, x_{n+1})=\underline{x}$. Let $U_{j}=\pi_{j}(D)$ and $\rho_{j}=\max_{\underline{x}\in\partial U_{j}}|\underline{x}-\underline{0}|$. Set $d=\min_{j}d_{j}$ and $\rho=\max_{j}\rho_{j}$.

To do prove (8.1), let $Q\in\partial D$ be temporarily fixed. Then $Q\in L_{j_{0}}\cap \partial D$ for some $1<j_{0}<s$, or possibly $Q\in L_{j^{\prime}_{0}}\cap \partial D$ with $|j^{\prime}_{0}-j_{0}|=1$. In fact, with respect to the latter case, $Q\in L_{j_{0}}\cap L_{j^{\prime}_{0}}\,(\neq\emptyset)$, and in the following argument, we only consider the latter case, so does the former case. Furthermore, it is easy to find that $\pi_{j_{0}}\big(B(Q, \frac{d}{2})\cap L_{j_{0}}\cap\partial D\big)\subset B_{j_{0}}(\underline{0}, \rho)$ and $\pi_{j^{\prime}_{0}}\big(B(Q, \frac{d}{2})\cap L_{j^{\prime}_{0}}\cap\partial D\big)\subset B_{j^{\prime}_{0}}(\underline{0}, \rho)$.

With the above preliminaries, by (8.3), we have
\begin{align}
\int_{\partial D}\big|\langle D_{m}(Q, P), n_{Q}\rangle\big| d\sigma(P)&\leq C_{m,n}\int_{\partial D}|P-Q|^{2m-(n+2)}\left(1+\big|\log|P-Q|\big|\right) d\sigma(P)\\
&\leq C_{m,n,\mathrm{diam} (D)}\int_{\partial D}|P-Q|^{2m-(n+2)-\eta} d\sigma(P)\nonumber\\
&\leq C_{m,n,\mathrm{diam} (D)}\int_{\partial D}\frac{1}{|P-Q|^{(n-2)+\eta}} d\sigma(P)\,\,(\mathrm{since}\,\,m\geq2)\nonumber\\
&=C_{m,n,\mathrm{diam} (D)}\Big[\int_{\partial D\cap B(Q, \frac{d}{2})}\frac{1}{|P-Q|^{(n-2)+\eta}} d\sigma(P)\nonumber\\
&\,\,\,\,\,\,+\int_{\partial D\setminus B(Q, \frac{d}{2})}\frac{1}{|P-Q|^{(n-2)+\eta}} d\sigma(P)\Big]\nonumber\\
&\leq C_{m,n,\mathrm{diam} (D)}\Big[\int_{\partial D\cap B(Q, \frac{d}{2})}\frac{1}{|P-Q|^{(n-2)+\eta}} d\sigma(P)\nonumber\\
&\,\,\,\,\,\,\,+\left(\frac{2}{d}\right)^{n-2+\eta}\int_{\partial D\setminus B(Q, \frac{d}{2})}d\sigma(P)\Big]\nonumber\\
&\leq C_{m,n,\mathrm{diam} (D)}\Big[\int_{\partial D\cap B(Q, \frac{d}{2})}\frac{1}{|P-Q|^{(n-2)+\eta}} d\sigma(P)\nonumber\\
&\,\,\,\,\,\,\,+\left(\frac{2}{d}\right)^{n-2+\eta}\sigma(\partial D)\Big]\nonumber
\end{align}
in which
\begin{align}
&\,\,\,\,\,\,\int_{\partial D\cap B(Q, \frac{d}{2})}\frac{1}{|P-Q|^{(n-2)+\eta}} d\sigma(P)\\
&\leq\int_{\partial D\cap L_{j_{0}}\cap B(Q, \frac{d}{2})}\frac{1}{|P-Q|^{(n-2)+\eta}} d\sigma(P)\nonumber\\
&\,\,\,\,\,\,+\int_{\partial D\cap L_{j^{\prime}_{0}}\cap B(Q, \frac{d}{2})}\frac{1}{|P-Q|^{(n-2)+\eta}} d\sigma(P)\nonumber\\
&=\int_{\pi_{j_{0}}\big(\partial D\cap L_{j_{0}}\cap B(Q, \frac{d}{2})\big)}\frac{\sqrt{1+|\nabla\varphi_{j_{0}}(\underline{x})|^{2}}}{\big(|\underline{x}-\underline{x_{Q}}|^{2}+|\varphi_{j_{0}}(\underline{x})-\varphi_{j_{0}}(\underline{x_{Q}})|^{2}\big)^{\frac{(n-2)+\eta}{2}}} d\underline{x}\nonumber\\
&\,\,\,\,\,\,+\int_{\pi_{j^{\prime}_{0}}\big(\partial D\cap L_{j^{\prime}_{0}}\cap B(Q, \frac{d}{2})\big)}\frac{\sqrt{1+|\nabla\varphi_{j_{0}^{\prime}}(\underline{x})|^{2}}}{\big(|\underline{x}-\underline{x_{Q}}|^{2}+|\varphi_{j_{0}^{\prime}}(\underline{x})-\varphi_{j^{\prime}_{0}}(\underline{x_{Q}})|^{2}\big)^{\frac{(n-2)+\eta}{2}}} d\underline{x}\nonumber\\
&\leq\sqrt{1+\mathcal{L}^{2}}\left[\int_{B_{j_{0}}(\underline{0}, \rho)}\frac{1}{|\underline{x}-\underline{x_{Q}}|^{(n-2)+\eta}} d\underline{x}+\int_{B_{j^{\prime}_{0}}(\underline{0}, \rho)}\frac{1}{|\underline{x}-\underline{x_{Q}}|^{(n-2)+\eta}} d\underline{x}\right]\nonumber\\
&\leq\sqrt{1+\mathcal{L}^{2}}\left[\int_{B_{j_{0}}(\underline{x_{Q}}, 2\rho)}\frac{1}{|\underline{x}-\underline{x_{Q}}|^{(n-2)+\eta}} d\underline{x}+\int_{B_{j^{\prime}_{0}}(\underline{x_{Q}}, 2\rho)}\frac{1}{|\underline{x}-\underline{x_{Q}}|^{(n-2)+\eta}} d\underline{x}\right]\nonumber
\end{align}
\begin{align}
&\leq2\sqrt{1+\mathcal{L}^{2}} \int_{0}^{2\rho}\int_{S^{n-1}}\frac{1}{r^{(n-2)+\eta}}r^{n-1}drd\sigma(\omega)\nonumber\\
&=\frac{2}{2-\eta}(2\rho)^{2-\eta}\sqrt{1+\mathcal{L}^{2}}\sigma(S^{n-1}),\nonumber
\end{align}
where $0<\eta<1$ which can be arbitrary selected, the fact $\lim_{|P-Q|\rightarrow 0}|P-Q|^{\eta}\log|P-Q|=0$ has been used in the second inequality of (8.4); whereas in the third inequality in (8.5), we have used the fact that $\underline{x},\, \underline{x_{Q}}\in B_{j_{0}}(\underline{0}, \rho)\,\big(B_{j^{\prime}_{0}}(\underline{0}, \rho)\big)$  implies $\underline{x}\in B_{j_{0}}(\underline{x_{Q}}, 2\rho)\,\big(B_{j^{\prime}_{0}}(\underline{x_{Q}}, 2\rho)\big)$.

Therefore, by (8.4) and (8.5), we have
\begin{align}
\int_{\partial D}\big|\langle D_{m}(Q, P), n_{Q}\rangle\big| d\sigma(P)&\leq C_{m,n,\mathrm{diam} D}\Big[\frac{2}{2-\eta}(2\rho)^{2-\eta}\sqrt{1+\mathcal{L}^{2}}\sigma(S^{n-1})\\
&\,\,\,\,\,\,+\left(\frac{2}{d}\right)^{n-2+\eta}\sigma(\partial D)\Big].\nonumber
\end{align}

Denote
\begin{equation}
C(m,n,D)=C_{m,n,\mathrm{diam} D}\left[\frac{2}{2-\eta}(2\rho)^{2-\eta}\sqrt{1+\mathcal{L}^{2}}\sigma(S^{n-1})+\left(\frac{2}{d}\right)^{n-2+\eta}\sigma(\partial D)\right],
\end{equation}
which depends only on $m, n$ and $D$, then (8.1) follows from (8.6) since $Q\in \partial D$ is arbitrarily chosen.. Thus the lemma is completed.
\end{proof}

\begin{lem}
Let $D$ be a bounded Lipschitz domain, $D_{m}=(D_{m}^{(1)}, \ldots, D_{m}^{(n+1)})$ in which $D_{m}^{(j)}$ are defined as in Lemma 2.2, then there exists a constant $C=C(m,n,D)$ such that
\begin{equation}
\sup_{X\in D}\left(\int_{\partial D}\big|\langle D_{m}(X, P), n_{P}\rangle\big| d\sigma(P)\right)<C
\end{equation}
and
\begin{equation}
\sup_{X\in D}\left(\int_{\partial D}\big|\langle D_{m}(X, P), n_{Q}\rangle\big| d\sigma(P)\right)<C
\end{equation}
for any $m\geq 2$, where $n_{P}$ and $n_{Q}$ are the outward unit normal vectors respectively at $P$ and $Q$ on $\in \partial D$.
\end{lem}

\begin{proof}
It is similar to Lemma 8.2.
\end{proof}

\begin{rem}
Let $D$ and $D_{m}$ be as above, by the above two lemmas or a direct argument, in fact, there exists a constant $C=C(m,n,D)$ such that
\begin{equation}
\sup_{X\in \overline{D}}\left(\int_{\partial D}\big|\langle D_{m}(X, P), n_{P}\rangle\big| d\sigma(P)\right)<C
\end{equation}
and
\begin{equation}
\sup_{X\in \overline{D}}\left(\int_{\partial D}\big|\langle D_{m}(X, P), n_{Q}\rangle\big| d\sigma(P)\right)<C
\end{equation}
for any $m\geq 2$, where $n_{P}$ and $n_{Q}$ are the outward unit normal vectors respectively at $P$ and $Q$ on $\in \partial D$.
\end{rem}

With $D_{m}$ replaced by $\mathcal{D}_{m}$, we also have

\begin{lem}
Let $D$ be a bounded Lipschitz domain, $\mathcal{D}_{m}$ are defined as in Lemma 5.1, then there exists a constant $C=C(m,n,D)$ such that
\begin{equation}
\sup_{Q\in\partial D}\left(\int_{\partial D} \big|\mathcal{D}_{m}(Q, P)\big|d\sigma(P)\right)<C
\end{equation}
for any $m\geq 2$.
\end{lem}

\begin{proof}
It is similar to Lemma 8.2.
\end{proof}

\begin{lem}
Let $D$ be a bounded Lipschitz domain, $\mathcal{D}_{m}$ are defined as in Lemma 5.1, then there exists a constant $C=C(m,n,D)$ such that
\begin{equation}
\sup_{X\in D}\left(\int_{\partial D} \big|\mathcal{D}_{m}(X, P)\big|d\sigma(P)\right)<C
\end{equation}
for any $m\geq 2$.
\end{lem}

\begin{proof}
It is similar to Lemma 8.5.
\end{proof}

\begin{rem}
Let $D$ and $\mathcal{D}_{m}$ be as above, by Lemmas 8.5 and 8.6 or a direct argument, in fact, we have that there exists a constant $C=C(m,n,D)$ such that
\begin{equation}
\sup_{X\in \overline{D}}\left(\int_{\partial D} \big|\mathcal{D}_{m}(X, P)\big|d\sigma(P)\right)<C
\end{equation}
for any $m\geq 2$.
\end{rem}

Furthermore, we have

\begin{lem}
Let $D$ be a bounded Lipschitz domain, $\mathcal{D}_{m}$ are defined as in Lemma 5.1, then there exists a constant $C=C(m,n,D)$ such that
\begin{equation}
\sup_{X\in \overline{D}}\left(\int_{\partial D} \big|\nabla\mathcal{D}_{m}(X, P)\big|d\sigma(P)\right)<C
\end{equation}
for any $m\geq 2$.
\end{lem}

\begin{proof}
By (5.17), $\nabla\mathcal{D}_{m}=D_{m}$. So it is similar to Lemma 8.2 as Remark 8.4 states.
\end{proof}

\begin{rem}
By observing the argument of Lemma 8.2, it is easy to find that Lemmas 8.5 and 8.6, as well as (8.14) in Remark 8.7 also holds when $m=1$.
\end{rem}

In terms of above lemmas, we can obtain some boundedness properties in $L^{p}$ for the operators $\mathrm{K}^{*}_{m}$, $\mathrm{K}_{m}$,  $M_{j}$, $\mathcal{M}_{j}$ and $\nabla \mathcal{M}_{j}$ and so on, which are important in the approach to solve the polyharmonic BVPs (1.1)-(1.3) in the case of bounded Lipschitz domains of this section.

\begin{thm}
Let $D$ be a bounded Lipschitz domain, $\mathrm{K}^{*}_{m}$, $m\geq 2$ be as in Theorem 6.3, then
\begin{equation}
\|\mathrm{K}^{*}_{m}f\|_{L^{p}(\partial D)}\leq C\|f\|_{L^{p}(\partial D)}
\end{equation}
for any $f\in L^{p}(\partial D)$, $1\leq p\leq \infty$. Furthermore, if
\begin{equation}
\int_{\partial D}\mathcal{N}_{m-1}(Q)f(Q)d\sigma(Q)=0,
\end{equation}
then
\begin{equation}
\int_{\partial D}\mathrm{K}^{*}_{m}f(P)d\sigma(P)=0,
\end{equation}
where $\mathcal{N}_{m-1}$ is the $(m-1)$-th order Newtonian potential on $D$ defined as follows
\begin{equation}
\mathcal{N}_{m-1}(Y)=\int_{D}\mathcal{D}_{m-1}(X, Y)dX,\,\,Y\in \mathbb{R}^{n+1}.
\end{equation}
\end{thm}

\begin{rem}
The classical Newtonian potential is referred to \cite{hel}.
\end{rem}

\begin{proof}
At first, it is easy to verify (8.16). In fact, by (8.1), $\mathrm{K}^{*}_{m}: L^{1}(\partial D)\rightarrow L^{1}(\partial D)$ is bounded. By (8.2), it is easily find that $\mathrm{K}^{*}_{m}: L^{\infty}(\partial D)\rightarrow L^{\infty}(\partial D)$ is also bounded. Then by the interpolation of operators, $\mathrm{K}^{*}_{m}: L^{p}(\partial D)\rightarrow L^{p}(\partial D)$ is bounded for $1<p<\infty$.

Next turn to (8.18) under (8.17). By the definition of the operator $\mathrm{K}^{*}_{m}$ and Theorem 5.4, we have
\begin{align}
\int_{\partial D}\mathrm{K}^{*}_{m}f(P)d\sigma(P)&=\int_{\partial D}\left[\int_{\partial D}\langle D_{m}(Q, P),
        n_{P}\rangle f(Q)d\sigma(Q)\right]d\sigma(P)\\
     &=\int_{\partial D}\left[\int_{\partial D}\langle D_{m}(Q, P),
        n_{P}\rangle d\sigma(P)\right]f(Q)d\sigma(Q)\nonumber\\
     &=\int_{\partial D}\left[\int_{\partial D}\langle \nabla\mathcal{D}_{m}(Q, P),
        n_{P}\rangle d\sigma(P)\right]f(Q)d\sigma(Q)\nonumber\\
     &=\int_{\partial D}\left[\int_{\partial D}\frac{\partial}{\partial N_{p}}\mathcal{D}_{m}(Q, P)d\sigma(P)\right]f(Q)d\sigma(Q)\nonumber
     \end{align}
where
\begin{align}
\int_{\partial D}\frac{\partial}{\partial N_{p}}\mathcal{D}_{m}(Q, P)d\sigma(P)&=\lim_{\epsilon\rightarrow 0}\int_{\partial D\setminus B(Q, \epsilon)}\frac{\partial}{\partial N_{p}}\mathcal{D}_{m}(Q, P)d\sigma(P)\\
&=\lim_{\epsilon\rightarrow 0}\left(\int_{\partial D\setminus B(Q, \epsilon)}+\int_{\partial D\cap B(Q, \epsilon)}\right)\frac{\partial}{\partial N_{p}}\mathcal{D}_{m}(Q, P)d\sigma(P)\nonumber\\
&=\lim_{\epsilon\rightarrow 0}\int_{D\setminus B(Q, \epsilon)}\mathrm{div}\nabla\big(\mathcal{D}_{m}(Q, X)\big)dX\nonumber\\
&=\int_{D}\Delta\mathcal{D}_{m}(Q, X)dX\nonumber\\
&=\int_{D}\mathcal{D}_{m-1}(Q, X)dX\nonumber\\
&=\mathcal{N}_{m-1}(Q)\nonumber
\end{align}
in which Gauass's divergence theorem, and the following easy facts are used (by Lebesgue's dominated convergence theorem, the details are similar to the argument of Lemma 8.2):
\begin{equation}
\lim_{\epsilon\rightarrow 0}\int_{\partial D\cap B(Q, \epsilon)}\frac{\partial}{\partial N_{p}}\mathcal{D}_{m}(Q, P)d\sigma(P)=0
\end{equation}
and
\begin{align}
\lim_{\epsilon\rightarrow 0}\int_{D\setminus B(Q, \epsilon)}\mathrm{div}\nabla\big(\mathcal{D}_{m}(Q, X)\big)dX&=\lim_{\epsilon\rightarrow 0}\int_{D\setminus B(Q, \epsilon)}\Delta\mathcal{D}_{m}(Q, X)dX\\
&=\int_{D}\Delta\mathcal{D}_{m}(Q, X)dX\nonumber.
\end{align}
Therefore, by (8.17), (8.20) and (8.21), we have
\begin{equation*}
\int_{\partial D}\mathrm{K}^{*}_{m}f(P)d\sigma(P)=\int_{\partial D}\mathcal{N}_{m-1}(Q)f(Q)d\sigma(Q)=0.
\end{equation*}
\end{proof}

\begin{thm}
Let $D$ be a bounded Lipschitz domain, and $\mathrm{K}_{m}$, $m\geq 2$ be the same as in Theorem 3.5, then $\mathrm{K}_{m}: L^{p}(\partial D)\rightarrow L^{p}(\partial D)$ is bounded for $1\leq p\leq\infty$.
\end{thm}

\begin{proof}
By duality in term of Theorem 8.10, or directly verify by a similar argument to Theorem 8.10 by invoking Lemma 8.2.
\end{proof}

\begin{thm}
Let $D$ be a bounded Lipschitz domain, and $M_{j}$, $j\geq 2$ be the $j$th layer $\mathcal{D}$-potential, then $M_{j}: L^{p}(\partial D)\rightarrow L^{p}(D)$ is bounded for $1\leq p\leq\infty$.
\end{thm}

\begin{proof}
By Lemma 8.3 and the Riesz-Thorin interpolation theorem of operators, it is similar to Theorem 8.10.
\end{proof}

\begin{thm}
Let $D$ be a bounded Lipschitz domain, and $\mathcal{M}_{j}$, $j\geq 1$ be the $j$th layer $\mathcal{S}$-potential, then $\mathcal{M}_{j}: L^{p}(\partial D)\rightarrow L^{p}(\partial D)$ is bounded for $1\leq p\leq\infty$.
\end{thm}

\begin{proof}
It is similar to Theorem 8.10 by using Lemma 8.5, the claims in Remark 8.9 and the interpolation of operators.
\end{proof}

\begin{thm}
Let $D$ be a bounded Lipschitz domain, and $\mathcal{M}_{j}$, $j\geq 1$ be the $j$th layer $\mathcal{S}$-potential, then $\mathcal{M}_{j}: L^{p}(\partial D)\rightarrow L^{p}(D)$ is bounded for $1\leq p\leq\infty$.
\end{thm}

\begin{proof}
It is similar to Theorem 8.10 by using Lemma 8.6, the claims in Remark 8.9 and the interpolation of operators.
\end{proof}

\begin{thm}
Let $D$ be a bounded Lipschitz domain, and $\mathcal{M}_{j}$, $j\geq 2$ be the $j$th layer $\mathcal{S}$-potential, then $\nabla\mathcal{M}_{m}: L^{p}(\partial D)\rightarrow L^{p}(D)$ is bounded for $1\leq p\leq\infty$.
\end{thm}

\begin{proof}
It is similar to Theorem 8.10 by using Lemma 8.8 and the interpolation of operators.
\end{proof}

\begin{rem}
By Lemma 8.8 and the statements in Remarks 8.4, 8.7and 8.9, in fact, by performing a similar argument to Theorem 8.10, we have that all the operators $M_{j}$ and $\nabla\mathcal{M}_{j}$ are bounded from $L^{p}(\partial D)$ to $L^{p}(\overline{D})$ for any $j\geq 2$ and $1\leq p\leq\infty$, whereas $\mathcal{M}_{j}: L^{p}(\partial D)\rightarrow L^{p}(\overline{D})$ is bounded for any $j\geq 1$ and $1\leq p\leq\infty$.
\end{rem}

The following lemma is crucial to the non-tangential maximal estimates of solutions for the $L^{p}$ polyharmonic BVPs discussing in this section, whose analogue is also significant to the corresponding estimates of the Dirichlet and Neumann problems in $L^{p}$ for Laplace's equation (see \cite{dk,dk1}).

\begin{thm}
Let $D$ be a bounded Lipschitz domain with the coordinate systems $(L_{j},Q_{j}), \varphi_{j}$ and $\pi_{j}$ as the same as in the proof of Lemma 8.2, $M_{m}$, $m\geq 1$ be the $j$th layer $\mathcal{D}$-potential. If $X\in L_{j_{0}}\cap D$ for some $1\leq j_{0}\leq s$, set $P\in \partial D\cap L_{j_{0}}$ with $\pi_{j_{0}}(X)=\pi_{j_{0}}(P)$, and $\rho=|X-P|$, then for any $f\in L^{p_{m}}(\partial D)$,
\begin{equation}
|M_{m}f(X)-(\mathrm{K}_{m})_{\rho}f(P)|\leq CM^{*}f(P)
\end{equation}
where
\begin{equation}
(\mathrm{K}_{m})_{\rho}f(P)=\int_{\partial D\setminus B_{\rho}(P)}\langle D_{m}(P, Q), n_{Q}\rangle f(Q)d\sigma(Q),
\end{equation}
the maximal function $M^{*}f$ is defined as follows
\begin{equation}
M^{*}f(P)=\sup_{r>0}\left[\frac{1}{\sigma\big(\partial D\cap B_{r}(P)\big)}\int_{\partial D\cap B_{r}(P)}|f(Q)|d\sigma(Q)\right],\,\,P\in \partial D
\end{equation}
and
\begin{equation}
p_{m}\in\begin{cases}
(1, \infty), \,\,m=1;\vspace{2mm}\\
[1, \infty], \,\,m\geq 2.
\end{cases}
\end{equation}
\end{thm}

\begin{proof}
It is due to Dahlberg in the case of $m=1$ (i.e., Proposition 1.1, \cite{dk}). To other cases, as the proof of Lemma 8.2, by invoking the local coordinates, it can be attained by a similar argument to Dahlberg's one.
\end{proof}

\begin{thm}
Let $D$ be a bounded Lipschitz domain, $(\mathrm{K}_{m})_{\rho}$ be defined as (8.25). For any $f\in L^{p_{m}}(\partial D)$, define
the maximal operator
\begin{equation}
\mathrm{K}_{m}^{\#}f(P)=\sup_{\rho>0}|(\mathrm{K}_{m})_{\rho}f(P)|, \,\,P\in\partial D,
\end{equation}
then
\begin{equation}
\|\mathrm{K}_{m}^{\#}f\|_{L^{p_{m}}(\partial D)}\leq C\|f\|_{L^{p_{m}}(\partial D)},
\end{equation}
where $p_{m}$ is given by (8.27), and $C$ is a constant depending only on $m,n,p_{m}$ and $D$.
\end{thm}

\begin{proof}
The case of $m=1$ is a deep and classical result \cite{dk,gr2,st2}. By Lemma 8.2 and the interpolation of operators, other cases follows.
\end{proof}

\begin{thm}
Let $D$ be a bounded Lipschitz domain, $M_{m}$, $m\geq 1$ be the $j$th layer $\mathcal{D}$-potential, then for any $f\in L^{p}(\partial D)$ with $1<p<\infty$,
\begin{equation}
\|M(M_{m}f)\|_{L^{p}(\partial D)}\leq C\|f\|_{L^{p}(\partial D)},
\end{equation}
where $M(\cdot)$ is the nontangential maximal function given by (1.4), and $C$ is a constant depending only on $m,n,p$ and $D$.
\end{thm}

\begin{proof}
Since $M^{*}: L^{p}(\partial D)\rightarrow L^{p}(\partial D)$ is bounded for any $1<p<\infty$ (e.g., see \cite{st2}), then by Theorems 8.18 and 8.19, (8.30) follows immediately. The case of $m=1$ is classical.
\end{proof}

However, the multi-layer $S$-potentials version of Lemmas 8.18-8.20 is the following

\begin{thm}
Let $D$ be a bounded Lipschitz domain with the coordinate systems $(L_{j},Q_{j}), \varphi_{j}$ and $\pi_{j}$ as the same as in the proof of Lemma 8.2, $\mathcal{M}_{m}$, $m\geq 1$ be the $j$th layer $\mathcal{S}$-potential. If $X\in L_{j_{0}}\cap D$ for some $1\leq j_{0}\leq s$, set $P\in \partial D\cap L_{j_{0}}$ with $\pi_{j_{0}}(X)=\pi_{j_{0}}(P)$, and $\rho=|X-P|$, then for any $f\in L^{p_{m}}(\partial D)$,
\begin{equation}
|\nabla\mathcal{M}_{m}f(X)-(\widetilde{\mathrm{K}}_{m})_{\rho}f(P)|\leq CM^{*}f(P)
\end{equation}
where
\begin{equation}
(\widetilde{\mathrm{K}}_{m})_{\rho}f(P)=\int_{\partial D\setminus B_{\rho}(P)}\nabla\mathcal{D}_{m}(P, Q)f(Q) d\sigma(Q),
\end{equation}
the maximal function $M^{*}f$ are defined by (8.26),
$\nabla$ is the gradient operator and
$p_{m}$ is given by (8.27).
\end{thm}

\begin{proof}
It is similar to Theorem 8.18.
\end{proof}

\begin{thm}
Let $D$ be a bounded Lipschitz domain, $(\widetilde{\mathrm{K}}_{m})_{\rho}$ be defined as (8.32). For any $f\in L^{p_{m}}(\partial D)$, set
the maximal operator
\begin{equation}
\widetilde{\mathrm{K}}_{m}^{\#}f(P)=\sup_{\rho>0}|(\widetilde{\mathrm{K}}_{m})_{\rho}f(P)|, \,\,P\in\partial D,
\end{equation}
then
\begin{equation}
\|\widetilde{\mathrm{K}}_{m}^{\#}f\|_{L^{p_{m}}(\partial D)}\leq C\|f\|_{L^{p_{m}}(\partial D)},
\end{equation}
where $p_{m}$ is given by (8.27), and $C$ is a constant depending only on $m,n,p_{m}$ and $D$.
\end{thm}

\begin{proof}
Similar to Theorem 8.19.
\end{proof}

\begin{thm}
Let $D$ be a bounded Lipschitz domain, $\mathcal{M}_{m}$, $m\geq 1$ be the $j$th layer $\mathcal{S}$-potential, then for any $f\in L^{p}(\partial D)$ with $1<p<\infty$,
\begin{equation}
\|M(\nabla\mathcal{M}_{m}f)\|_{L^{p}(\partial D)}\leq C\|f\|_{L^{p}(\partial D)},
\end{equation}
where $\nabla$ is the gradient operator, $M(\cdot)$ is the nontangential maximal function given by (1.4), and $C$ is a constant depending only on $m,n,p$ and $D$.
\end{thm}

\begin{proof}
Similar to Theorem 8.20.
\end{proof}

Now we can give the main results in this section as follows

\begin{thm}
Let $\{\,K_{m}\,\}_{m=1}^{\infty}$ be the sequence of the Poisson
fields, and $D$ be a bounded Lipschitz domain in
$\mathbb{R}^{n+1}$ with boundary $\partial D$, then
for any $m>1$, there exists $\varepsilon=\varepsilon(D)>0$ such that
the PHD problem (4.1) with the data $f_{j}\in L^{p}(\partial D)$,
$2-\varepsilon<p<\infty$, is solvable and a
solution is given by
\begin{align}
u(X)&=\sum_{j=1}^{m}\int_{\partial D}\langle K_{j}(X, Q),
n_{Q}\rangle \widetilde{f}_{j-1}(Q)d\sigma(Q),\\
&=\sum_{j=1}^{m}M_{j}\widetilde{f}_{j-1}(X),\,\,\,\, X\in D,\nonumber
\end{align}
where
\begin{equation}\widetilde{f}_{m-1}=\left(\frac{1}{2}I+T\right)^{-1}f_{m-1}\end{equation} and
\begin{equation}\widetilde{f}_{l}=\left(\frac{1}{2}I+T\right)^{-1}\left(f_{l}-\sum_{j=l+2}^{m}\mathrm{K}_{j-l}\widetilde{f}_{j-1}\right)\end{equation}
with $0\leq l\leq m-2$, which satisfying the following estimates
\begin{equation}
\|u-M_{1}\widetilde{f}_{0}\|_{L^{p}(D)}\leq C\sum_{j=1}^{m-1}\|f_{j}\|_{L^{p}(\partial D)}
\end{equation}
and
\begin{equation}
\|M(u)\|_{L^{p}(\partial D)}\leq C\sum_{j=0}^{m-1}\|f_{j}\|_{L^{p}(\partial D)}
\end{equation}
in which $M(u)$ is the non-tangential maximal function of $u$ on $\partial D$.
Under any of the above two estimates, the solution (8.36) with (8.37) and (8.38) is unique.
\end{thm}

\begin{proof}
It is similar to Theorem 4.5 by using Lemma 3.4, Theorems 8.12, 8.13 and 8.20.
\end{proof}

\begin{thm}
Let $\{\,\mathcal{K}_{m}\,\}_{m=1}^{\infty}$ be the sequence of the
polyharmonic fundamental solutions, and $D$ be a bounded Lipschitz
domain in $\mathbb{R}^{n+1}$ with boundary $\partial
D$, then for any $m>1$, there exists $\varepsilon=\varepsilon(D)>0$
such that the PHN problem (6.1) with the data $g_{m-1}\in
L_{0}^{p}(\partial D)$, $g_{j}\in
L^{p}(\partial D)$, $0\leq j\leq m-2$, $1<p<2+\varepsilon$, is
solvable and a solution is given by
\begin{align}
u(X)&=\sum_{j=1}^{m}\int_{\partial D}\mathcal{K}_{j}(X,
Q)\widetilde{g}_{j-1}(Q)d\sigma(Q),\\
&=\sum_{j=1}^{m}\mathcal{M}_{j}\widetilde{g}_{j-1}(X),\,\, X\in D,\nonumber
\end{align}
where
\begin{equation}\widetilde{g}_{m-1}=\left(-\frac{1}{2}I+T^{*}\right)^{-1}g_{m-1}\end{equation}
and
\begin{equation}\widetilde{g}_{l}=\left(-\frac{1}{2}I+T^{*}\right)^{-1}\left(g_{l}+\sum_{j=l+2}^{m}\mathrm{K}^{*}_{j-l}\widetilde{g}_{j-1}\right)\end{equation}
with $0\leq l\leq m-2$, which satisfying the following estimates
\begin{equation}
\|\nabla(u-\mathcal{M}_{1}\widetilde{g}_{0})\|_{L^{p}(D)}\leq C\sum_{j=1}^{m-1}\|g_{j}\|_{L^{p}(\partial D)},
\end{equation}
\begin{equation}
\|u\|_{L^{p}(D)}\leq C\sum_{j=0}^{m-1}\|g_{j}\|_{L^{p}(\partial D)}
\end{equation}
and
\begin{equation}
\|M(\nabla u)\|_{L^{p}(\partial D)}\leq C\sum_{j=0}^{m-1}\|g_{j}\|_{L^{p}(\partial D)}
\end{equation}
in which $M(\nabla u)$ is the non-tangential maximal function of $\nabla u$ on $\partial D$.
The solution (8.41) with (8.42) and (8.43) is unique under (8.45), and unique up to a constant under (8.44) and (8.46).
\end{thm}

\begin{proof}
It is similar to Theorem 6.7 by noting Remark 8.9 and using Lemmas 3.4 and 8.1, Theorems 8.10, 8.15, 8.16 and 8.23.
\end{proof}

\begin{rem}
By the second claim in Theorem 8.10, if
\begin{equation}
\int_{\partial D}\mathcal{N}_{l}\widetilde{g}_{j}d\sigma=0, \,\, 1\leq j\leq m-1\,\, and \,\, 1\leq l\leq j,
\end{equation}
where $\mathcal{N}_{l}$ is the $l$th order Newtonian potential defined in (8.23), then
\begin{equation}
\int_{\partial D}\mathrm{K}^{*}_{l+1}\widetilde{g}_{j}d\sigma=0, \,\, 1\leq j\leq m-1\,\, and \,\, 1\leq l\leq j.
\end{equation}
Therefore, by Lemma 8.1, (8.42) and (8.43), we obtain that $\widetilde{g}_{j}\in L^{p}_{0}(\partial D)$, and further that $g_{j}\in L^{p}_{0}(\partial D)$, $0\leq j\leq m-2$.
\end{rem}

\begin{thm}
Let $\{\,\mathcal{K}_{m}\,\}_{m=1}^{\infty}$ be the sequence of the
polyharmonic fundamental solutions, and $D$ be a bounded Lipschitz
domain in $\mathbb{R}^{n+1}$ with boundary $\partial
D$, then for any $m>1$, there exists $\varepsilon=\varepsilon(D)>0$
such that the PHR problem (7.1) with the data $h_{j}\in
L_{1}^{p}(\partial D)$, $0\leq j<m$, $1<p<2+\varepsilon$, is
solvable and a solution is given by
\begin{align}
u(X)&=\sum_{j=1}^{m}\int_{\partial D}\mathcal{K}_{j}(X,
Q)\widetilde{h}_{j-1}(Q)d\sigma(Q),\\
&=\sum_{j=1}^{m}\mathcal{M}_{j}\widetilde{h}_{j-1}(X),\,\, X\in D,\nonumber
\end{align}
where
\begin{equation}\widetilde{h}_{m-1}=\mathcal{M}_{1}^{-1}h_{m-1}\end{equation}
and
\begin{equation}\widetilde{h}_{l}=\mathcal{M}_{1}^{-1}\left(h_{l}-\sum_{j=l+2}^{m}\mathcal{M}_{j-l}\widetilde{h}_{j-1}\right)\end{equation}
with $0\leq l\leq m-2$, which satisfying the following estimates
\begin{equation}
\|\nabla(u-\mathcal{M}_{1}\widetilde{h}_{0})\|_{L^{p}(D)}\leq C\sum_{j=1}^{m-1}\|h_{j}\|_{L_{1}^{p}(\partial D)},
\end{equation}
\begin{equation}
\|u\|_{L^{p}(D)}\leq C\sum_{j=0}^{m-1}\|h_{j}\|_{L_{1}^{p}(\partial D)}
\end{equation}
and
\begin{equation}
\|M(\nabla u)\|_{L^{p}(\partial D)}\leq C\sum_{j=0}^{m-1}\|h_{j}\|_{L_{1}^{p}(\partial D)}
\end{equation}
in which $M(\nabla u)$ is the non-tangential maximal function of $\nabla u$ on $\partial D$.
The solution (8.49) with (8.50) and (8.51) is unique under (8.53), and unique up to a constant under (8.52) and (8.54).
\end{thm}

\begin{proof}
It is similar to Theorem 7.2 by noting Remark 8.9 and invoking Lemma 7.1, Theorems 8.14-8.16, and 8.23.
\end{proof}


\end{document}